\documentclass[11pt]{article}

\usepackage{amsmath,amsfonts,amsthm,amscd,amssymb,graphicx,mathtools}
\usepackage{cite}
\usepackage{bbm}
\usepackage{color}
\usepackage[english=american]{csquotes}
\usepackage[colorlinks=true, pdfstartview=FitV, linkcolor=black, citecolor=black, urlcolor=black]{hyperref}
\usepackage{calc}
\usepackage{mathptmx}
\usepackage{t1enc}
\usepackage{stackrel}
\usepackage{dsfont}
\usepackage{enumerate}
\usepackage{hyperref}
\usepackage{wasysym}
\usepackage[T1]{fontenc}

\newtheorem{theo}{Theorem}[section]
\newtheorem{lem}[theo]{Lemma}
\newtheorem{prop}[theo]{Proposition}

\newtheorem{example}[theo]{Example}

\newtheorem{rem}[theo]{Remark}

\DeclareMathOperator{\dist}{dist}

\DeclareMathOperator{\id}{id}

\DeclareMathOperator{\diverg}{div}
\DeclareMathOperator{\curl}{curl}
\DeclareMathOperator{\rank}{rank}

\DeclareMathOperator{\supp}{supp}

\newcommand{\T}{\mathbb{T}}

\newcommand{\N}{\mathbb{N}}
\newcommand{\R}{\mathbb{R}}

\newcommand{\Q}{\mathbb{Q}}
\newcommand{\Z}{\mathbb{Z}}

\newcommand{\eps}{\epsilon}

\newcommand{\dx}{ \mathrm{d}x}
\newcommand{\dt}{\mathrm{d}t}
\newcommand{\ds}{\mathrm{d}s}
\newcommand{\dy}{\mathrm{d}y}
\newcommand{\dz}{\mathrm{d}z}
\newcommand{\dd}{\mathrm{d}}
\newcommand{\e}{\mathrm{e}}
\newcommand{\image}{\mathrm{im}\;}

\renewcommand{\eps}{\varepsilon}
\renewcommand{\epsilon}{\varepsilon}

 \def\Xint#1{\mathchoice
 	{\XXint\displaystyle\textstyle{#1}}%
 	{\XXint\textstyle\scriptstyle{#1}}%
 	{\XXint\scriptstyle\scriptscriptstyle{#1}}%
 	{\XXint\scriptscriptstyle\scriptscriptstyle{#1}}%
 	\!\int}
 \def\XXint#1#2#3{{\setbox0=\hbox{$#1{#2#3}{\int}$ }
 		\vcenter{\hbox{$#2#3$ }}\kern-.58\wd0}}
 
 \def\dashint{\Xint-}

\numberwithin{equation}{section}
\DeclareMathAlphabet{\mathcal}{OMS}{cmsy}{m}{n}

\begin{document}
\allowdisplaybreaks
	\title{Which measure-valued solutions of the monoatomic gas equations are generated by weak solutions?}
	
	\author{Dennis Gallenm\"uller\footnotemark[1]\ \ and Emil Wiedemann\footnotemark[2]}
	
	\date{}
	
	\maketitle
	
	\begin{abstract}
		Contrary to the incompressible case not every measure-valued solution of the compressible Euler equations can be generated by weak solutions or a vanishing viscosity sequence. In the present paper we give sufficient conditions on an admissible measure-valued solution of the isentropic Euler system to be generated by weak solutions. As one of the crucial steps we prove a characterization result for generating $\mathcal{A}$-free Young measures in terms of potential operators including uniform $L^{\infty}$-bounds. More concrete versions of our results are presented in the case of a solution consisting of two Dirac measures. We conclude by discussing also necessary conditions for generating a measure-valued solution by weak solutions or a vanishing viscosity sequence and will point out that the resulting gap mainly results from obtaining only uniform $L^p$-bounds for $1<p<\infty$ instead of $p=\infty$.		
	\end{abstract}
	
	\renewcommand{\thefootnote}{\fnsymbol{footnote}}
	
	\footnotetext[1]{Institute of Applied Analysis, Universit\"at Ulm, Helmholtzstra\ss e 18, 89081 Ulm, Germany. Email: dennis.gallenmueller@uni-ulm.de}
	
	\footnotetext[2]{Department of Mathematics, Friedrich-Alexander-Universit\"at Erlangen-N\"urnberg, Cauerstra\ss e 11, 91058 Erlangen, Germany. Email: emil.wiedemann@fau.de}
	
\section{Introduction}
	
	Weak concepts of solution are a standard tool in modern PDE theory to handle the issue of existence of solutions. Understanding the structure and the qualitative behavior of such solutions then leads at best to uniqueness. The aim of the present work is to relate the notions of distributional and measure-valued solutions for compressible fluid flows in order to obtain certain criteria for discarding unphysical solutions.\\
	In particular, we will consider the isentropic Euler system in momentum form
	\begin{equation}
		\begin{aligned}
			\partial_tm+\operatorname{div}_x\left(\frac{m\otimes m}{\rho} \right)+\nabla p(\rho)&=0,\\
			\partial_t\rho+\operatorname{div}_xm&=0\label{eq:euler}
		\end{aligned}
	\end{equation}
	over $[0,T]\times \T^d$ or $[0,T]\times \T^d$ for some fixed final time $T>0$ on the torus $\T^d$ or some open set $\Omega\subset \R^d$ in $d\geq 2$ dimensions, where the unknowns are the density $\rho$ and the momentum $m$. The pressure $p$ is predetermined as a function of $\rho$ by a constitutive relation. Throughout this paper we will only consider the choice
	\begin{align}
		p(\rho)=\rho^{\gamma}\text{ with }\gamma=1+\frac{2}{d}\label{eq:gammadefinition}
	\end{align}
	of a monoatomic gas. The reason for this lies in the equality of the energy and the generalized pressure up to multiplication with a constant in this case.\\
	A first criterion for a solution to be physical is that the total energy should be non-increasing. More concretely, we may assume that the \textit{total energy density} $e(\rho,m)=\frac{1}{\gamma-1}\rho^{\gamma}+\frac{|m|^2}{2\rho}$ satisfies the inequality
	\begin{align}
		\int\limits_{\Omega}^{}e(\rho,m)(t,x)\dx\leq \int\limits_{\Omega}^{}e(\rho_0,m_0)(x)\dx\label{eq:energyinequalitydef}
	\end{align}
	for a.e.~$t\in (0,T)$, where $(\rho_0,m_0)$ is the initial data at time $t=0$. Solutions possessing this property will be denoted as \textit{energy admissible} or simply \textit{admissible}. We may note here that it is also common to use a local energy inequality
	\begin{align*}
		\partial_te(\rho,m)+\operatorname{div}_x\left((e(\rho,m)+\rho^{\gamma})\frac{m}{\rho}\right)\leq 0.
	\end{align*}
	It turns out to be quite challenging to include the latter condition in our approach since we are not able to freely choose the energy profile of the solutions in our proofs. We thus restrict ourselves to the consideration of admissibility in the integral sense formulated in (\ref{eq:energyinequalitydef}) similarly as in the recent contribution \cite{CVY21}.
	It is still unknown if admissible distributional solutions, which we will henceforth only call admissible \textit{weak solutions}, exist for the isentropic Euler equations for any initial data. On the other hand, by the method of convex integration introduced for the incompressible Euler equations in the seminal work of De Lellis and Sz\'ekelyhidi, cf.~\cite{DLSz9,DLSz10}, it can be shown that there is initial data for the isentropic Euler system giving rise to infinitely many admissible weak solutions, see for example \cite{DLSz10} and \cite{C14}. Although weak solutions already incorporate low-regularity behavior such as shocks and turbulence effects, we will introduce the even more general notion of \textit{measure-valued solutions}. Such measure-valued solutions are weakly* measurable functions $(t,x)\mapsto \nu_{(t,x)}$ with values in the space of probability measures. So, every weak solution corresponds to a measure-valued solution consisting only of a Dirac measure. However, despite this apparently very weak behavior, measure-valued solutions enjoy the property of weak-strong uniqueness, cf.~\cite{GSGW15}. There is also strong numerical evidence supporting the use of measure-valued solutions, see~\cite{FMT}.\\
	Originally, the concept of measure-valued solution was introduced for the incompressible Euler system by DiPerna and Majda \cite{DM87} to capture oscillation and concentration effects. Note that we will only work with uniformly bounded functions, hence concentrations are not possible. The adaptation to the compressible case and the proof of existence for any integrable initial data has been carried out by Neustupa \cite{N93}. It is not hard to show that an admissible \textit{vanishing viscosity sequence} $(\rho_{\mu},m_{\mu})$, i.e.~a sequence of admissible weak solutions of the compressible Navier-Stokes system with viscosity $\mu$ tending to zero, which is uniformly bounded with density bounded away from zero and initial data independent of $\mu$, will generate an admissible measure-valued solution as $\mu$ goes to zero. Although it is not automatic from the definition that every measure-valued solution is obtained from such a sequence, one is tempted to ask if this may in fact hold true. A related question is whether every admissible measure-valued solution is generated by a sequence of admissible weak solutions of the Euler equations. Surprisingly, for the incompressible Euler system the latter is actually the case, see \cite{SW12}. However, the recent works \cite{CFKW17} and \cite{GW20} show that the isentropic Euler system admits measure-valued solutions that cannot be generated by a vanishing viscosity limit or even a sequence of weak solutions. The solutions constructed in the latter articles consist of two Dirac measures supported on weak solutions. In fact, for such solutions one can formulate a selection principle: If there exists a uniformly bounded generating vanishing viscosity sequence or sequence of weak solutions for a diatomic measure-valued solution then the underlying weak solutions have to be wave-cone connected, as stated in \cite{GW20}. One can argue that a solution which cannot be generated by a vanishing viscosity sequence should be discarded as unphysical, cf.~\cite{BTW12}. We will study in the present paper conditions under which the wave-cone connectedness of the underlying weak solutions implies indeed the existence of a generating sequence of weak solutions.\\
	The case of diatomic measures will be an application of the main objective of this paper: We give sufficient conditions for generating a measure-valued solution by weak solutions. Our first main result is the following theorem. Note that we will assume throughout the present paper that the space-dimension $d$ is strictly greater than one. 
	\begin{theo}\label{theo:weakmainresult}
		Let $T>0$. Let $\nu$ be a measure-valued solution of the isentropic Euler system with initial data $(\rho_0,m_0)\in L^{\infty}(\T^d)$. Suppose $\nu$ fulfills the following conditions:
		\begin{itemize}
			\item There exists some $\eta>0$ such that
			\begin{align*}
				\supp\left(\nu_{(t,x)}\right)\subset\left\{(\rho,m)\in\R^+\times \R^d\,:\, \rho\geq \eta\right\}
			\end{align*}
			for a.e.~$(t,x)\in (0,T)\times\T^d$.
			\item There exists some $R>0$ such that
			\begin{align*}
				\langle\tilde{\nu}_{(t,x)},f\rangle\geq Q^{R}_{\mathcal{B}_E}f(\langle\tilde{\nu}_{(t,x)},\operatorname{id}\rangle)
			\end{align*}
			for a.e.~$(t,x)\in (0,T)\times\T^d$ and all $f\in C(\R^+\times \R^d\times S_0^d\times \R^+)$.
			\item The barycenter of the lift satisfies
			\begin{align*}
				\langle\tilde{\nu},\operatorname{id}\rangle=\sigma+\mathcal{B}_Ew
			\end{align*}
			for some $w\in W^{2,\infty}((0,T)\times \T^d)$ and some $\sigma\in C([0,T]\times \T^d,\R^m)$.
		\end{itemize}
		Then $\nu$ is generated by a uniformly bounded sequence of weak solutions $(\rho_j,m_j)$ such that
		\begin{align*}
			\rho_j&\geq \tilde{\eta}\text{ for a.e.~}(t,x)\in(0,T)\times \T^d,\\
			\rho_j(0,\cdot)&\rightharpoonup \rho_0\text{ in }L^{\gamma}(\T^d),\\
			m_j(0,\cdot)&\rightharpoonup m_0\text{ in }L^{2}(\T^d)
		\end{align*}
		for some $\tilde{\eta}>0$.
	\end{theo}
	The notation we used here is introduced in Sections \ref{sect:preliminaries} and \ref{sect:fonsecamueller}.	In words our result means that essentially three conditions are sufficient for generating a measure-valued solution by weak solutions:\\
	First, we need to exclude vacuum and bound our density away from zero. Secondly, the \textit{lifted measure} $\tilde{\nu}$ satisfies a Jensen-type inequality involving the \textit{truncated quasiconvex envelope} $Q_{\mathcal{B}_E}^R$ with truncation threshold $R>0$, where $\mathcal{B}_E$ is the potential operator, constructed in \cite{G20}, associated to the relaxed Euler system, which we rewrite also as a linear homogeneous differential operator $\mathcal{A}_E$. Here we use the framework of $\mathcal{A}$-free Young measures introduced by Murat and Tartar, see e.g.~\cite{M78,T79}. This kind of Jensen inequality is inspired by the fundamental papers \cite{KP91,KP94} of Kinderlehrer and Pedregal, where this condition represents the main characterizing feature of a Young measure to be generated by a sequence with a certain structure, in their case a sequence of gradients. This involves so-called \textit{quasiconvex} functions. One may consult \cite{Mueller99} for an overview on quasiconvexity and gradient Young measures. The more general concept of $\mathcal{A}$\textit{-quasiconvexity} has been introduced by Dacorogna \cite{D82} and developed further by Fonseca and Müller \cite{FM99}. In our situation, the Jensen condition is actually the main ingredient to ensure the existence of an $L^{\infty}$-bounded sequence of \textit{subsolutions}, i.e.~solutions to the linearly relaxed Euler system. Note that we consider only oscillation Young measures due to the overall uniform bounds on the generating sequences. For more on $\mathcal{A}$-free generability of Young measures including also concentrations see e.g.~\cite{A19,ADR20,AS21,KR19}. The third condition in Theorem~\ref{theo:weakmainresult} that the barycenter of the lift is in potential form modulo an additive continuous function is quite natural, since the barycenter is a subsolution by assumption. Note that the assumptions associated to uniform $L^{\infty}$-bounds, namely the truncation in the quasiconvex envelope and the potential of the barycenter lying in $W^{2,\infty}$, are the only unnatural conditions, as we will see when we discuss necessary conditions, cf.~Remark~\ref{rem:comparison}. However, these uniform bounds are needed for technical reasons in the proof.\\	
	If we consider energy admissible measure-valued solutions, we are able to generate them by likewise admissible weak solutions. This is our second main result.
	\begin{theo}\label{theo:mainresult}
		Let $T>0$. Let $\nu$ be an admissible measure-valued solution of the isentropic Euler system with initial data $(\rho_0,m_0)\in L^{\infty}(\T^d)$. Suppose $\nu$ fulfills the same conditions as in Theorem \ref{theo:weakmainresult} and additionally:
		\begin{itemize}
			\item It holds that
			\begin{align*}
				\supp\left(\tilde{\nu}_{(t,x)}\right)\subset\left\{(\rho,m,M,Q)\,:\, \rho\geq \eta\text{ and }Q=\frac{2}{d}\langle\nu_{(t,x)},e\rangle\right\}
			\end{align*}
			for a.e.~$(t,x)\in (0,T)\times\T^d$.
			\item The map $(t,x)\mapsto \langle\nu_{(t,x)},e\rangle$ is continuous on $[0,T]\times \T^d$.
		\end{itemize}
		Then $\nu$ is generated by a uniformly bounded sequence of admissible weak solutions $(\rho_j,m_j)$ such that for all $j\in\N$ it holds that
		\begin{align*}
			\|\rho_j(t=0)-\rho_0\|_{L^{\gamma}(\T^d)}&\leq \frac{1}{j},\\
			\|m_j(t=0)-m_0\|_{L^{2}(\T^d)}&\leq \frac{1}{j},\\
			\rho_j&\geq \tilde{\eta}\text{ for a.e.~}(t,x)\in(0,T)\times \T^d
		\end{align*}
		for some $\tilde{\eta}>0$.
	\end{theo}
	\begin{rem}
		Compared to Theorem \ref{theo:weakmainresult} we needed to additionally assume the quite restrictive conditions that the generalized pressure component $Q$ of a point in the support of the lift equals the energy density of the measure and that the energy density is continuous.\\
		Note that it is straightforward to check that we could formulate the concentration condition on the generalized pressure for the support of the lifted measure $\supp\left(\tilde{\nu}_{(t,x)} \right)\subset \left\{(\rho,m,M,Q)\,:\,Q=\frac{2}{d}\langle\nu_{(t,x)},e\rangle \right\}$ equivalently only in terms of the original measure $\nu$ as a concentration in the energy
		\begin{align*}
			\supp(\nu_{(t,x)})\subset \left\{(\rho,m)\,:\, e(\rho,m)= C_{\nu_{(t,x)}} \right\}
		\end{align*}
		for some fixed constant $C_{\nu_{(t,x)}}$ depending only on the measure $\nu_{(t,x)}$ at the point $(t,x)$. In this case, clearly $C_{\nu_{(t,x)}}=\langle\nu_{(t,x)},e\rangle$ holds. 
	\end{rem}
	Let us give a brief sketch of the proof of our main results:\\
	In the first step we obtain a sequence of functions in potential form with uniformly bounded second order gradients, which generates the lifted measure. To achieve this, we will prove a characterization result in the spirit of Fonseca and M\"uller \cite{FM99}, cf.~Theorem~\ref{theo:fromsubsolntosequence}. Our specific characterization result stems from the particular definition of the truncated quasiconvex envelope. The resulting sequence of potentials converges only in measure to a specific convex set $K$, which we will need later for applying the convex integration. The second step then improves this convergence in measure to uniform convergence to $K$ by a truncation result from \cite{G20}. The latter truncation method is a technique developed from previous results of Zhang \cite{Z92} and Müller \cite{M99}.	Crucial ingredients at this point are the $L^{\infty}$-boundedness of the second order gradients of the generating sequence and the fact that the potential $\mathcal{B}_E$ is of order two. In the third step we mollify our sequence of subsolutions, which will provide us with the required regularity for the convex integration scheme. As a fourth step we modify the generating sequence at small times by using the perturbation property of convex integration, see Proposition \ref{timezeroperturbation} below, in order to guarantee energy admissibility. Here, the specific form of the support of the lifted measure comes into play. The last step then uses convex integration to transition from the sequence of subsolutions generating the lifted measure to a sequence of weak solutions generating the measure-valued solution. For the latter we will make use of a compressible version of convex integration from \cite{EJT20}, see Theorem \ref{theo:EJT} below, which is a generalization to non-constant energies of the results in \cite{M20}. Note that the only parts in which the particular form of the Euler equations plays a role are this last convex integration step and the step using the perturbation property. The assumption of the dimension being greater than one is used exactly here, since convex integration needs more than one space dimension.\\
	\begin{rem}\label{rem:constantrank}
		Note that we do not need an explicit constant rank condition here, since we already assume to work on the level of potentials with uniform bounds on the highest order gradients. Nevertheless, for coming up with examples it will be important to show that $\mathcal{A}_E$ and $\mathcal{B}_E$ have constant rank, cf.~Section~\ref{sect:application}. For $\mathcal{A}_E$, this is straightforward to check, see e.g.~Lemma~1 in \cite{CFKW17} for $d=3$. In the case $d=2$ considered in Section~\ref{sect:application}, we also show that $\mathcal{B}_E$ has constant rank, see Lemma~\ref{lem:eulerpotentialproperties}, below. For higher dimensions though, one needs to compute the operator $\mathcal{B}_E$ first, which for $d=3$ already consists of a $(36\times 10)$-matrix that has to be determined from the implicit algorithm given in \cite{G20}. Hence, a proof of the constant rank property of $\mathcal{B}_E$ for dimensions $d\geq 3$ is still missing, although we presume that the assertion is correct.\\
		Note, however, that constant rank conditions are not necessary for the study of differential constraints, see e.g.~\cite{Mue99}.
	\end{rem}
	\begin{rem}
		Since the Young measure $\nu$ is generated in both Theorems \ref{theo:weakmainresult} and \ref{theo:mainresult} by a uniformly bounded sequence of functions, standard Young measure theory implies that a posteriori the support of $\nu$ is compact. This observation applies of course also to all subsequent results.\\
		In the previous sketch of the proof we have seen that the truncation method from \cite{G20} is crucial. For this a uniform bound on the second order gradients $(D^2u_j)$ of the generating sequence $(u_j)$ is required. Moreover, one also obtains from this truncation a sequence $(g_j)$ with the same structure. It follows that the resulting sequence $(\mathcal{B}_Eg_j)$ is again uniformly bounded, which therefore implies that all Young measures considered had to have compact support. Note that the uniform bounds simplify the problem significantly and might therefore be too restrictive. However, going beyond uniformly bounded sequences seems to be quite difficult, see also Remark~\ref{rem:comparison}, below.\\
		In the recent work \cite{BGS21} also other $L^{\infty}$-truncation techniques for $\mathcal{A}_E$ have been considered, which do not require any a priori bounds on the gradients. However, unfortunately the resulting sequence from the latter truncation scheme is only uniformly bounded, but one does not obtain information about convergence to an arbitrarily fixed convex set $K$ as we need it for the application of convex integration. Hence, we cannot simply overcome the overall assumption of uniformly bounded highest order gradients. For more on $L^{\infty}$-truncation of divergence-free sequences see also \cite{S22}.\\
		There is also a slightly different way of proving the Müller-Zhang truncation used in our proof. For that one first shows the truncation result for the second order gradient $u\mapsto D^2u$ without uniform bounds on $\|D^2u_j\|$ by following Müller's approach in \cite{M99} more directly. This is possible, since $u\mapsto D^2u$ already covers all directions. By multiplying with a suitable coefficient matrix $B_E$ one obtains $\mathcal{B}_E = B_E\cdot D^2$. Hence, one might wonder if this implies the truncation result for $\mathcal{B}_E$ without uniform bounds. However, for the transition from $(D^2u_j)$ to $(\mathcal{B}_Eu_j)$, again, we need the uniform boundedness of $(D^2u_j)$, see Remark~3.4 in \cite{G20}.
	\end{rem}
	We will also give necessary conditions for generating measure-valued solutions by either weak solutions or a vanishing viscosity sequence. Previously, Chiodaroli et al.~\cite{CFKW17} already gave a result with such necessary conditions, mainly discussing a related Jensen-type inequality. In the present work we adapt this to our more specific situation of $L^{\infty}$-bounded functions and develop the improvements resulting from these bounds. Comparing the necessary conditions with the sufficient conditions from Theorem~\ref{theo:weakmainresult} we observe that the emerging gap basically lies in the requirement of $L^{\infty}$-bounds for sufficiency and the lack of those in the necessity, where only $L^p$-bounds for $1<p<\infty$ can be obtained.\\
	Let us conclude this introduction by briefly giving some comments on the difficulties that occurred in our investigations for the compressible case contrasting the incompressible case. The convex integration method for the compressible Euler equations yields the equality of the energy density and the generalized pressure $Q$, which itself is a part of the subsolution. This restriction does not exist in the incompressible case. Thus, we are not able to freely choose the energy distribution of our solutions which also led us to the use of the total energy inequality as admissibility criterion instead of the local energy inequality. Note also that the associated wave-cone is much smaller in the compressible case. This then corresponds to a non-trivial Jensen-type condition. Moreover, the $\Lambda$-convex hull $K$, where our subsolutions have to be contained in for the convex integration, is more complicated. While in the incompressible case, the $\Lambda$-convex hull of the constitutive set is the whole phase space, in the isentropic situation the states in the hull have to satisfy 
	\begin{equation*}
		Q\geq p(\rho)+\frac{2}{d}e_{kin}(\rho,m,M),
	\end{equation*}
	(see Sections~\ref{sect:preliminaries} and~\ref{sect:proofofthemainresults} for the notation), and in order to guarantee that the subsolutions have this property almost everywhere, the above mentioned truncation techniques had to be applied. \\
	The outline of the paper is as follows: In Section \ref{sect:preliminaries} we present the notation and some preliminaries on homogeneous differential operators needed throughout the proofs. Section \ref{sect:fonsecamueller} is dedicated to the proof of the characterization result for generating Young measures by sequences of potentials with $L^{\infty}$-bounded highest order gradients. Here, we also introduce and discuss some properties of the truncated quasiconvex envelope. In Section~\ref{sect:weakandmvs} we give the definitions of weak and measure-valued solutions and provide some preliminary results on those. The main part of the paper is then contained in Section \ref{sect:proofofthemainresults} where we give a proof of the main results Theorem~\ref{theo:weakmainresult} and Theorem~\ref{theo:mainresult}. The application of our sufficient conditions to diatomic measure-valued solutions then is the subject of Section \ref{sect:application}. Finally, in Section \ref{sect:necessaryconditions} we give also some necessary conditions and a brief discussion of the differences between those and the sufficient conditions. We will also describe the technical difficulties that would have to be overcome in order to give a complete characterization of compactly supported measure-valued solutions.

\section{Preliminaries}\label{sect:preliminaries}
\subsection{Notation}
Throughout this paper we will use the following definitions and notation. The dimension $d$ of the underlying space(-time) will be a natural number with $d\geq 2$. We will write $\mathcal{Q}:=(0,1)^{d}$ for the unit cube, $\mathbb{S}^{d-1}$ for the $(d-1)$-dimensional unit sphere, and $S_0^d$ for the space of trace-free $(d\times d)$-matrices. Note that we denote the $d$-dimensional flat torus by $\T^d$, which can be viewed as the unit cube $\mathcal{Q}$ with periodic boundaries. Moreover, for vectors $v,w\in\R^d$ the matrix denoted by $v\otimes w$ consists of the entries $(v\otimes w)_{i,j}=v_iw_j$. For $v\in \R^d$ we also introduce the shorthand notation
\begin{align*}
	v\ocircle v=v\otimes v-\frac{|v|^2}{d}\mathbb{E}_d.
\end{align*}
Here $\mathbb{E}_d$ is the $d$-dimensional identity matrix. Furthermore, $|A|$ denotes the operator norm for any matrix $A$.\\
Let $\mathcal{P}$ denote the set of probability measures over a given state space. We will call a weakly*-measurable map $\nu\colon \Omega\to \mathcal{P}(\R^m)$ a Young measure over $\Omega\subset \R^d$ open. One also writes $\nu\in L^{\infty}_{\operatorname{w}}(\Omega,\mathcal{P}(\R^m))$. Let $\nu$ be a Young measure and let $(w_n)$ be a sequence of measurable functions. We say that $(w_n)$ generates $\nu$, denoted by $w_n\overset{Y}{\rightharpoonup}\nu$, if
\begin{align*}
	\int\limits_{\Omega}^{}\int\varphi(x)f(w_n(x))\dx\rightarrow \int\limits_{\Omega}^{}\varphi(x)\int\limits_{\R^m}^{}f(z)\dd\nu_{x}(z)\dx
\end{align*}
for all $\varphi\in L^1(\Omega)$ and $f\in C_0(\R^m)$. See e.g. \cite{Ri18} for a more detailed exposition of Young measure theory.
\subsection{Homogeneous Differential Operators}	
	Let $\mathcal{A}=\sum\limits_{|\alpha|=k}^{}A^{\alpha}\partial_{\alpha}$ be a linear, homogeneous differential operator over $\R^d$ of order $k\in\N$ with constant coefficients $A^{\alpha}\in\R^{n\times m}$. In particular, $\mathcal{A}$ can be viewed as an operator from $C^{\infty}(\R^d,\R^m)$ to $C^{\infty}(\R^d,\R^n)$. A linear, homogeneous differential operator $\mathcal{B}=\sum\limits_{|\alpha|=l}^{}B^{\alpha}\partial_{\alpha}$ with constant coefficients $B^{\alpha}\in\R^{m\times N}$ is called \textit{potential} for $\mathcal{A}$ if
	\begin{align}
		\ker\mathbb{A}(\xi)=\image\mathbb{B}(\xi)\text{ for all }\xi\in\R^d\backslash\{0\}.\label{eq:potentialdefinition}
	\end{align}
	Here $\mathbb{A}$ and $\mathbb{B}$ denote the Fourier symbols of $\mathcal{A}$ and $\mathcal{B}$, respectively. This terminology is adopted from \cite{R18}. Furthermore, $\mathcal{A}$ is said to have \textit{constant rank} if there exists some $r\in\N$ such that
	\begin{align*}
		\rank\mathbb{A}(\xi)=r\text{ for all }\xi\in\R^d\backslash\{0\}.
	\end{align*}
	In fact, Theorem 1 in \cite{R18} shows that a homogeneous linear differential operator $\mathcal{A}$ has constant rank if and only if there exists a potential $\mathcal{B}$ for $\mathcal{A}$ that has constant rank.\\
	The following result characterizes the relation between operator and potential by their action on smooth periodic functions in the case of constant ranks.
	\begin{prop}\label{prop:potentialcharacterization}
		Let $\mathcal{A}$ and $\mathcal{B}$ be homogeneous constant rank operators. Then $\mathcal{B}$ is a potential for $\mathcal{A}$ in the sense of (\ref{eq:potentialdefinition}) if and only if for all $z\in C^{\infty}(\T^{d})\cap \ker\mathcal{A}$ with $\dashint z=0$ there exists some $u\in C^{\infty}(\T^{d})$ such that $z=\mathcal{B}u$ and for all $u\in C^{\infty}(\T^{d})$ it holds that $\mathcal{A}\mathcal{B}u=0$.
	\end{prop}
	\begin{proof}
		By using Lemma~3.5 from \cite{AS21} it only remains to show that for constant rank operators it holds that
		\begin{align*}
			\ker\mathbb{A}(\omega)=\image\mathbb{B}(\omega)\text{ for all }\omega\in\R^d\backslash\{0\}\text{ for all }\xi\in\R^d\backslash\{0\}
		\end{align*}
		if
		\begin{align*}
			\ker\mathbb{A}(\omega)=\image\mathbb{B}(\omega)\text{ for all }\omega\in\R^d\backslash\{0\}\text{ for all }\xi\in\Z^d\backslash\{0\}.
		\end{align*}
		Indeed, assume the latter. Then for $\omega\in \Q^{d}\backslash \{0\}$ observe that $\lambda\omega\in\Z^{d}\backslash\{0\}$, where $\lambda$ denotes the least common multiple of the denominators of the components of $\omega$. Then by the homogeneity of $\mathcal{A}$ and $\mathcal{B}$ we obtain $\ker\mathbb{A}(\omega)=\image\mathbb{B}(\omega)$.\\
		Let now $\omega\in \R^d\backslash\{0\}$. Choose a sequence $(\omega_n)\subset \Q^{d}\backslash\{0\}$ with $\omega_n\rightarrow\omega$. As $\mathcal{A}$ and $\mathcal{B}$ have constant rank, Theorem 4.2 in \cite{R97} implies that $\mathbb{A}^{\dagger}(\omega_n)\rightarrow \mathbb{A}^{\dagger}(\omega)$ and $\mathbb{B}^{\dagger}(\omega_n)\rightarrow \mathbb{B}^{\dagger}(\omega)$, where $\mathbb{A}^{\dagger}$ and $\mathbb{B}^{\dagger}$ are the Moore-Penrose pseudoinverses of $\mathbb{A}$ and $\mathbb{B}$, respectively. Now let $w\in\R^N$. Then
		\begin{align*}
			\mathbb{A}(\omega)\mathbb{B}(\omega)w\leftarrow \mathbb{A}(\omega_n)\mathbb{B}(\omega_n)w=0,
		\end{align*}
		hence $\image\mathbb{B}(\omega)\subset \ker\mathbb{A}(\omega)$. On the other hand, for $z\in\ker\mathbb{A}(\omega)$ we have
		\begin{align*}
			z=\operatorname{pr}_{\ker\mathbb{A}(\omega)}z=(1-\mathbb{A}^{\dagger}(\omega)\mathbb{A}(\omega))z\leftarrow (1-\mathbb{A}^{\dagger}(\omega_n)\mathbb{A}(\omega_n))z=\operatorname{pr}_{\ker\mathbb{A}(\omega_n)}z=\mathbb{B}(\omega_n)w_{n},
		\end{align*}
		where we chose $w_{n}\in (\ker\mathbb{B}(\omega_n))^{\perp}$ such that $\operatorname{pr}_{\ker\mathbb{A}(\omega_n)}z=\mathbb{B}(\omega_n)w_n$. This is possible as we already know that $\ker\mathbb{A}(\omega_n)=\image\mathbb{B}(\omega_n)$. We obtain
		\begin{align*}
			|w_n-w_{m}|\leq |\mathbb{B}^{\dagger}(\omega_n)-\mathbb{B}^{\dagger}(\omega_m)| |\mathbb{B}(\omega_n)w_{n}|+|\mathbb{B}^{\dagger}(\omega_m)| |\mathbb{B}(\omega_n)w_{n}-\mathbb{B}(\omega_m)w_{m}|,
		\end{align*}
		which tends to zero as $m,n\rightarrow\infty$. Thus, $(w_n)$ is Cauchy and converges to some $w\in \R^N$. This implies that $z=\mathbb{B}(\omega)w$. Hence, also the inclusion $\ker\mathbb{A}(\omega)\subset \image\mathbb{B}(\omega)$ holds.
	\end{proof}
	Let us conclude this section with the following well-known result from the theory of Young measures.
	\begin{lem}\label{lem:measureproperties}
		Let $\Omega\subset \R^d$ be open. Let $(w_n)\subset L^1(\Omega)$ and $w\in L^1(\Omega)$. Suppose $w_n\overset{Y}{\rightharpoonup}\nu$ for some Young measure $\nu\colon \Omega\to\mathcal{P}(\R^m)$. Then $w_n+w\overset{Y}{\rightharpoonup}\tau_{w}\nu$ with $\tau_{w}\nu$ defined as
		\begin{align*}
			\langle(\tau_{w}\nu)_y,f\rangle:=\int\limits_{\R^m}^{}f(z+w(y))\dd\nu_y(z)
		\end{align*}
		for $f\in C_0(\R^m)$ and a.e.~$y\in\Omega$.
	\end{lem}
\section{A characterization theorem in the spirit of Fonseca-M\"uller}\label{sect:fonsecamueller}
	Throughout this section let $\mathcal{B}\colon C^{\infty}(\R^d,\R^N)\to C^{\infty}(\R^d,\R^m)$ be a linear homogeneous differential operator of order $l\in\N$. Define for $q>0$, $z\in \R^m$ and $g\in C(\R^m)$ the \textit{truncated quasiconvex envelope}
	\begin{align*}
		Q^{q}_{\mathcal{B}}g(z):=\inf\left\{\int\limits_{\T^{d}}^{}g(z+\mathcal{B}w)\dd x\,:\,w\in C_c^{\infty}(\mathcal{Q}),\  \|D^lw\|_{L^{\infty}}\leq q \right\}.
	\end{align*}
	Note that this terminology is not standard. In light of Corollary 1 in \cite{R18} this can be viewed as a truncated version of the $\mathcal{A}$\textit{-quasiconvex envelope}, cf.~\cite{FM99} for a detailed investigation of the latter concept.\\
	Let us study some preliminary properties of the truncated quasiconvex envelope that will be used in Section \ref{sect:application}. First, we will show that $Q_{\mathcal{B}}^q$ is strongly continuous in the truncation bound $q$.	
	\begin{lem}\label{lem:qlimit}
		Let $f\in C(\R^m)$ and let $q>0$. Then for all $z\in \R^m$ we have
		\begin{align*}
			\lim\limits_{a\rightarrow 0}Q_{\mathcal{B}}^{q+a}f(z)=Q_{\mathcal{B}}^qf(z).
		\end{align*}
	\end{lem}
	\begin{proof}
		Fix $f\in C(\R^m)$ and $q>0$. On the one hand, for all $-q<a<q$ we have
		\begin{align*}
			Q_{\mathcal{B}}^{q+|a|}f(z)\leq Q_{\mathcal{B}}^qf(z)\leq Q_{\mathcal{B}}^{q-|a|}f(z)
		\end{align*}
		for all $z\in \R^m$. On the other hand, let $-q<a<q$ and let $\delta>0$. For all $z\in \R^m$ we can choose some $u_{\delta,z}^{q+|a|}\in C_c^{\infty}(\mathcal{Q})$ with $\|D^lu_{\delta,z}^{q+|a|}\|_{L^{\infty}}\leq q+|a|$, approximating the infimum in the definition of $Q_{\mathcal{B}}^{q+|a|}f(z)$ with error at most $\delta$. Denote the modulus of continuity of $f$ on $B_{|z|+2C_\mathcal{B}q}(0)$ by $\omega_{f,z,q}$, where $C_{\mathcal{B}}$ is a constant depending only on the operator $\mathcal{B}$ satisfying $\|\mathcal{B}h\|_{L^{\infty}}\leq C_{\mathcal{B}}\|D^lh\|_{L^{\infty}}$ for all $h\in C^{\infty}(\R^d)$. Then we estimate
		\begin{align*}
			Q_{\mathcal{B}}^{q+|a|}f(z)&\geq \int\limits_{\T^{d}}^{}f\left(z+\mathcal{B}u_{\delta,z}^{q+|a|}(x) \right)\dx-\delta\\
			&\geq \int\limits_{\T^{d}}^{}f\left(z+\mathcal{B}\left(u_{\delta,z}^{q+|a|}(x)-\frac{|a|}{q+|a|}\cdot u_{\delta,z}^{q+|a|}(x) \right) \right)\dx-\delta-\omega_{f,z,q}(2|a|C_{\mathcal{B}})\\
			&\geq Q_{\mathcal{B}}^{q}f(z)-\delta-\omega_{f,z,q}(2|a|C_{\mathcal{B}}).
		\end{align*}
		Similarly, we estimate
		\begin{align*}
			Q_{\mathcal{B}}^{q}f(z)&\geq \int\limits_{\T^{d}}^{}f\left(z+\mathcal{B}u_{\delta,z}^{q}(x) \right)\dx-\delta\\
			&\geq Q_{\mathcal{B}}^{q-|a|}f(z)-\delta-\omega_{f,z,q}(2|a|C_{\mathcal{B}}).
		\end{align*}
		Letting $\delta\rightarrow 0$ and $a\rightarrow 0$ yields that
		\begin{align*}
			Q_{\mathcal{B}}^{q}f(z)\leq \lim\limits_{a\rightarrow 0}Q_{\mathcal{B}}^{q+|a|}f(z)\leq Q_{\mathcal{B}}^{q}f(z)\leq \lim\limits_{a\rightarrow 0}Q_{\mathcal{B}}^{q-|a|}f(z)\leq Q_{\mathcal{B}}^{q}f(z),
		\end{align*}
		which implies the assertion.
	\end{proof}
	Clearly we have $f(z)=Q^0_\mathcal{B}f(z)\geq Q^a_\mathcal{B}f(z)$. It is straightforward to show that the previous lemma can be extended to the case $q=0$.\\
	The following result states that the truncated quasiconvex envelope satisfies a kind of convexity condition on the wave-cone of $\mathcal{A}$.
	\begin{lem}\label{lem:waveconeconvexity}
		Assume that $\mathcal{B}$ is a constant rank operator of order $l$. For all $x_1,x_2\in \R^m$ satisfying $x_1-x_2\in \image\mathbb{B}(\omega)$ for some $\omega\in\mathbb{S}^{d-1}$ there exists a number $C_{\mathcal B}$ independent of $\omega$ such that
		\begin{align*}
			\lambda f(x_1)+(1-\lambda)f(x_2)\geq Q_{\mathcal{B}}^{C_{\mathcal B}|x_1-x_2|}f(\lambda x_1+(1-\lambda)x_2)
		\end{align*}
		for all $\lambda\in (0,1)$ and for all $f\in C(\R^m)$.
	\end{lem}
	\begin{proof}
		Let $\lambda\in (0,1)$ and $f\in C(\R^m)$. Fix two vectors $x_1,x_2$ such that $x_1-x_2\in \image \mathbb{B}(\omega)$ for some $\omega\in \mathbb{S}^{d-1}$. Let $\delta>0$ be smaller than $\frac{\lambda}{6}$. Actually $\delta$ will tend to zero in the end.\\
		The proof is based on the correct choice of test function. To this end, consider
		\begin{align*}
			\chi\colon [0,1)\to\R,\ x\mapsto\begin{cases}
				1-\lambda&\ x\in \left[0,{\lambda} \right),\\
				-\lambda& x\in \left[{\lambda},1 \right).
			\end{cases}
		\end{align*}
		Without renaming we extend $\chi$ periodically to $\R$. Then for all $n\in \N$ it holds that $z_n:=(x_1-x_2)\chi(nx\cdot \omega)$ satisfies $\|z_n\|_{L^{\infty}(\T^{d})}\leq |x_1-x_2|$ and moreover by standard arguments from e.g.~the appendix of \cite{BM84} there exists $n_{\omega,\delta}$ such that
		\begin{align*}
			\lambda f(x_1)+(1-\lambda)f(x_2)\geq \int\limits_{\T^{d}}^{}f(\lambda x_1+(1-\lambda)x_2+z_n(x))\dx-\delta
		\end{align*}
		for all $n\geq n_{\omega,\delta}$.
		It is not difficult to see that there exists a unique periodic function $\psi\in W^{l,\infty}(\mathbb T)$ of mean zero such that $D^l\psi=\chi$. Moreover, $\|D^j\psi\|_{L^\infty}\leq 1$ for all $j=1,\dots,l$.	
		By assumption there exists $\xi\in \R^N$ such that $x_1-x_2=\mathbb{B}(\omega) \xi=\mathbb{B}(\omega)\operatorname{pr}_{(\ker\mathbb{B}(\omega))^\perp}(\xi)$. Hence,
		\begin{align*}
			\mathcal{B}\left(\operatorname{pr}_{(\ker\mathbb{B}(\omega))^\perp}(\xi) n^{-l}\psi(nx\cdot\omega)\right)=\mathbb{B}(\omega)\xi \chi(nx\cdot\omega)=z_n(x)
		\end{align*}
		for all $n\in\N$, i.e.~$u_n(x):=\operatorname{pr}_{(\ker\mathbb{B}(\omega))^\perp}(\xi) n^{-l}\psi(nx\cdot\omega)$ is a potential for $z_n$. Let $\eta$ denote a standard mollifier on $\R^{d}$ and $\eta_{\frac{\delta}{n}}$ its rescaling with support in $B_{\frac{\delta}{n}}(0)$. Since $\mathcal{B}\left(u_n\ast\eta_{\frac{\delta}{n}}\right)=z_n\ast \eta_{\frac{\delta}{n}}=(z_1\ast\eta_{\delta})(n\bullet)$ equals $z_n$ on a subset of $\T^{d}$ of measure $(1-5\delta)$ for $n$ large enough and since $\left\|z_n\ast\eta_{\frac{\delta}{n}}\right\|_{L^{\infty}}\leq \|z_1\|_{L^{\infty}}$, we can estimate
		\begin{align*}
			\int\limits_{\T^{d}}^{}f(\lambda x_1+(1-\lambda)x_2+z_n(x))\dx\geq \int\limits_{\T^{d}}^{}f\left(\lambda x_1+(1-\lambda)x_2+\mathcal{B}\left(u_n\ast\eta_{\frac{\delta}{n}}\right)\right)\dx-\delta C_{f,x_1,x_2}
		\end{align*}
		for $n\geq n_{\omega,\delta}$ with $n_{\omega,\delta}$ large enough.\\
		In the final step we need to cut-off $(u_n\ast\eta_{\delta})\in C^{\infty}(\R^d)$ to ensure that its support is compactly contained in $\mathcal{Q}$. For that choose $\varphi_{\delta}\in C_c^{\infty}(\mathcal{Q})$ such that $0\leq \varphi_{\delta}\leq 1$, $\varphi_{\delta}(x)=1$ if $\dist(x,\partial \mathcal{Q})>\delta$, and $\|D^j\varphi_{\delta}\|_{L^{\infty}(\mathcal{Q})}\leq C\delta^{-j}$ for $j\in \{1,\dots,l\}$. Define
		\begin{align*}
			\tilde{u}_{\delta,n}(x):=\left(u_n\ast\eta_{\frac{\delta}{n}}\right)(x)\cdot\varphi_{\delta}(x).
		\end{align*}
		Then $\tilde{u}_{\delta,n}\in C_c^{\infty}(\mathcal{Q})$ and
		\begin{align*}
			\|D^l\tilde{u}_{\delta,n}\|_{L^{\infty}}&\leq \|D^lu_n\|_{L^{\infty}}+C_l\sum\limits_{j=0}^{l-1}\|D^ju_n\|_{L^{\infty}}\|D^{l-j}\varphi_{\delta}\|_{L^{\infty}}\\
			&\leq (1+C_l)|\mathbb{B}^{\dagger}(\omega)\mathbb{B}(\omega)\xi|\left(\|\chi\|_{L^{\infty}}+\sum\limits_{j=0}^{l-1}\frac{\left\|D^j\psi\right\|_{L^{\infty}}}{m^{l-j}\delta^{l-j}} \right)\\
			&\leq C_{\mathcal{B}}|x_1-x_2|\left(1+\frac{1}{\sqrt{n}} \right)
		\end{align*}
		for $n\geq n_{\omega,\delta}$ with $n_{\omega,\delta}$ large enough. Here, we used the fact that the pseudoinverse $\omega\mapsto \mathbb{B}^{\dagger}(\omega)$ is continuous on the compact set $\mathbb{S}^{d-1}$ due to the constant rank condition on $\mathcal{B}$, cf.~Theorem 4.2 in \cite{R97}. Thus,
		\begin{align*}
			\int\limits_{\T^{d}}^{}f\left(\lambda x_1+(1-\lambda)x_2+\mathcal{B}\left(u_n\ast\eta_{\frac{\delta}{n}}\right)\right)\dx&\geq \int\limits_{\T^{d}}^{}f\left(\lambda x_1+(1-\lambda)x_2+\mathcal{B}\tilde{u}_{\delta,n}\right)\dx-\delta C_{\mathcal{B},f,x_1,x_2}\\
			&\geq Q_{\mathcal{B}}^{C_{\mathcal{B}}|x_1-x_2|\left(1+\frac{1}{\sqrt{n}} \right)}f(\lambda x_1+(1-\lambda)x_2)-\delta C_{\mathcal{B},f,x_1,x_2}.
		\end{align*}
		In conclusion, with the help of Lemma \ref{lem:qlimit} we obtain the assertion if we let $n\rightarrow\infty$ and then $\delta\rightarrow 0$.
	\end{proof}
	We now come to the main result of this section. This result is a characterization theorem for Young measures being generated by uniformly bounded $\mathcal{A}$-free functions on the level of the potential $\mathcal{B}$ in the spirit of Theorem 4.1 in \cite{FM99}, the proof of which will be a guideline for our proof. Note that $\mathcal{B}$ being a potential for some operator $\mathcal{A}$ is not important as we formulate the theorem already on the level of $\mathcal{B}$. It is also worth mentioning that our theorem treats images of $\mathcal{B}$ with uniformly bounded highest order gradients which is the main reason we introduced the truncated quasiconvex envelope.
	\begin{theo}\label{theo:fromsubsolntosequence}
		Let $\Omega$ be open, bounded. Let $\nu$ be a Young measure such that $\langle\nu,\operatorname{id}\rangle\in L^1(\Omega)$. Let $R>0$. The following are equivalent:
		\begin{enumerate}
			\item[(i)] For all $f\in C(\R^m)$ and a.e.~$x\in\Omega$ it holds that
			\begin{align*}
				\langle\nu_{x},f\rangle&\geq Q^{R}_{\mathcal{B}}f(\langle\nu_{x},\operatorname{id}\rangle).
			\end{align*}
		\item[(ii)] There exists a sequence $(u_j)\subset C_c^{\infty}(\Omega)$ such that
		\begin{align*}
			\underset{j\rightarrow\infty}{\limsup}\|D^lu_j\|_{L^{\infty}(\Omega)}\leq & R,\\
			D^lu_j\overset{L^1(\Omega)}{\rightharpoonup}& 0,\\
			\mathcal{B}u_j+\langle \nu,\id\rangle\overset{Y}{\rightharpoonup}& \nu.
		\end{align*}
		\end{enumerate}
		The assertion of the theorem holds also for $\Omega$ replaced by the torus $\T^{d}$ or $(0,T)\times \T^{d-1}$ for some $T>0$.
	\end{theo}
	\begin{proof}
		As the proof for the torus $\T^{d}$ or $(0,T)\times \T^{d-1}$ as underlying domain are just a special cases of that for $\Omega\subset \R^d$ open, we only treat the latter in full detail.\\
		\\
		\textbf{Statement (i) implies (ii):}\\
		We follow the general strategy of the proof of Theorem 4.1 in \cite{FM99}, but adapt this to our specific situation at several points.\\
		Define
		\begin{align*}
			\mathbb{H}^{q}:=\Big\{&\nu\in\mathcal{P}(\R^m)\,:\,\langle\nu,\operatorname{id}\rangle=0\text{ and there exists }(w_j)\subset C_c^{\infty}(\mathcal{Q})\text{ with }\mathcal{B}w_j\overset{Y}{\rightharpoonup}\nu,\\
			& D^lw_j\overset{L^1(\T^{d})}{\rightharpoonup}0,\ \underset{j\rightarrow\infty}{\limsup}\|D^lw_j\|_{L^{\infty}(\T^{d})}\leq q \Big\},
		\end{align*}
		where $q>0$.\\
		\\
		\textbf{Claim 1:} $\mathbb{H}^{q}$ is convex.\\
		Let $\theta\in(0,1)$ and let $\nu,\mu\in\mathbb{H}^{q}$ with corresponding sequences $(v_j),(w_j)\subset C_c^{\infty}(\mathcal{Q})$ satisfying $\underset{j\rightarrow\infty}{\limsup}\|D^lv_j\|_{L^{\infty}(\T^{d})}\leq q$, and $D^lv_j\overset{L^1(\T^{d})}{\rightharpoonup}0$, similarly for $(w_j)$. By assumption we have $\|v_j\|_{W^{l,\infty}(\T^{d})}+\|w_j\|_{W^{l,\infty}(\T^{d})}\leq C$ for some $C>0$. Thus, after taking a subsequence we have $v_j\overset{*}{\rightharpoonup}v$ and $w_j\overset{*}{\rightharpoonup}w$ in $W_0^{l,\infty}(\mathcal{Q})$ for some $v,w\in W_0^{l,\infty}(\mathcal{Q})$. As $D^lv_j,D^lw_j$ tend to zero weakly in $L^1$, we have $v=w=0$. Sobolev embedding then yields $v_j\rightarrow 0$ and $w_j\rightarrow 0$ in $C^{l-1,\alpha}(\T^{d})$ for any $\alpha\in [0,1)$.\\
		For all $\varphi\in C_c^{\infty}(\mathcal{Q})$ we obtain
		\begin{align}
			\|\varphi\|_{C^l(\T^{d})}\|w_j-v_j\|_{C^{l-1}(\T^{d})}\rightarrow 0.\label{eq:interpolationconvergence}
		\end{align}
		Now choose a sequence $(\varphi_k)\subset C_c^{\infty}(\mathcal{Q})$ with $0\leq\varphi\leq 1$ and $\varphi_k\rightarrow \mathds{1}_{(0,\theta)\times \T^{d-1}}$ in $L^1(\T^{d})$. Using (\ref{eq:interpolationconvergence}) we can choose a subsequence of $v_j$ and $w_j$ such that
		\begin{align}
			\|\varphi_j\|_{C^l(\T^{d})}\|w_j-v_j\|_{C^{l-1}(\T^{d})}\leq \frac{1}{j}.\label{eq:varphiestimate}
		\end{align}
		Define
		\begin{align*}
			u_j:=v_j+\varphi_j\cdot(w_j-v_j)\in C_c^{\infty}(\mathcal{Q}).
		\end{align*}
		Let $\delta>0$ and choose $x_j\in\T^{d}$ such that
		\begin{align*}
			\|D^lv_j+\varphi_j(D^lw_j-D^lv_j)\|_{L^{\infty}(\T^{d})}\leq |D^lv_j+\varphi_j(D^lw_j-D^lv_j)|(x_j)+\delta.
		\end{align*}
		For $j$ large enough such that $\|D^lv_j\|_{L^{\infty}(\T^{d})}\leq q+\delta$ and $\|D^lw_j\|_{L^{\infty}(\T^{d})}\leq q+\delta$ it holds that
		\begin{align*}
			\|D^lu_j\|_{L^{\infty}(\T^{d})}\leq &\|D^lv_j+\varphi_j\cdot (D^lw_j-D^lv_j)\|_{L^{\infty}(\T^{d})}+\|\varphi_j\|_{C^l(\T^{d})}\|w_j-v_j\|_{C^{l-1}(\T^{d})}\\
			\leq &|D^lv_j+\varphi_j(D^lw_j-D^lv_j)|(x_j)+\delta+\frac{1}{j}\\
			\leq & q+2\delta+\frac{1}{j}.
		\end{align*}
		As $\delta>0$ was arbitrary, this yields $\underset{j\rightarrow\infty}{\limsup}\|D^lu_j\|_{L^{\infty}(\T^{d})}\leq q$. It is straightforward to check that $D^l(u_j-w_j)\rightarrow 0$ in $L^1((0,\theta)\times \T^{d-1})$ and $D^l(u_j-v_j)\rightarrow 0$ in $L^1((\theta,1)\times \T^{d-1})$. Hence, $D^lu_j\rightharpoonup 0$ in $L^1(\T^{d})$ and
		\begin{align*}
			\mathcal{B}u_j\overset{Y}{\rightharpoonup} \left(\lambda_{x}\right)_{x\in\Omega}:=\begin{cases}
			\mu,\, & x^1\in (0,\theta),\\
			\nu,\, & x^1\in (\theta,1).
			\end{cases}
		\end{align*}
		Now consider the sequence
		\begin{align*}
			\bar{u}_{j,m}(x):=\frac{1}{m^l}u_j(mx),\,m\in\N,
		\end{align*}
		which again lies in $C_c^{\infty}(\mathcal{Q})$. Then $\|D^l\bar{u}_{j,m}\|_{L^{\infty}(\T^{d})}=\|D^lu_j\|_{L^{\infty}(\T^{d})}$. Moreover, for $\psi\in L^1(\T^{d})$ and $g\in C_0(\R^m)$ we have using Proposition 2.8 in \cite{FM99}
		\begin{align*}
			\lim\limits_{j\rightarrow\infty}\lim\limits_{m\rightarrow\infty}\int\limits_{\T^{d}}^{}\psi(x)g(\mathcal{B}\bar{u}_{j,m}(x))\dx&=\lim\limits_{j\rightarrow\infty}\int\limits_{\T^{d}}^{}\psi(x)\dx\int\limits_{\T^{d}}^{}g(\mathcal{B}u_j(y))\dy\\
			&=(\theta\langle\mu,g\rangle+(1-\theta)\langle\nu,g\rangle)\int\limits_{\T^{d}}^{}\psi(x)\dx
		\end{align*}
		and for $\chi\in L^{\infty}(\T^{d})$ we have
		\begin{align*}
			\lim\limits_{j\rightarrow\infty}\lim\limits_{m\rightarrow\infty}\int\limits_{\T^{d}}^{}\chi(x)D^l\bar{u}_{j,m}(x)\dx=\lim\limits_{j\rightarrow\infty}\int\limits_{\T^{d}}^{}\chi(x)\dx\int\limits_{\T^{d}}^{}D^lu_j(y)\dy=0
		\end{align*}
		as $D^lu_j\rightharpoonup 0$ in $L^1(\T^{d})$. So, taking a diagonal subsequence yields that $\theta\mu+(1-\theta)\nu\in\mathbb{H}^{q}$.\\
		\\
		\textbf{Claim 2:} $\mathbb{H}^{q}$ is relatively closed in $\mathcal{P}(\R^m)$ with respect to the weak*-topology in $\mathcal{M}(\R^m)$.\\
		So, let $\nu\in \overline{\mathbb{H}^{q}}^{\mathcal{M}(\R^m)}\subset\mathcal{P}(\R^m)$ and let $(f_i)\subset L^1(\T^{d})$, $(g_j)\subset C_0(\R^m), (h_n)\subset L^1\big(\T^{d},\R^{d^l\times N}\big)$ be dense, countable subsets. For all $1\leq j\leq k$ there exist $\nu_k\in\mathbb{H}^{q}$ such that
		\begin{align}
			|\langle \nu-\nu_k,g_j\rangle|\leq \frac{1}{k}.\label{eq:nuconvergence}
		\end{align}
		Thus, for all $0\leq i,j,n\leq k$ there exist $w^k\in C_c^{\infty}(\T^{d})$ with $\|D^lw^k\|_{L^{\infty}(\T^{d})}\leq q+\frac{1}{k}$, and
		\begin{align*}
			\left|\langle\nu_k,g_j\rangle\int\limits_{\T^{d}}^{}f_i\,\dx-\int\limits_{\T^{d}}^{}f_ig_j(\mathcal{B}w^k)\dx \right|&\leq \frac{1}{k},\\
			\left|\int\limits_{\T^{d}}^{}h_n: D^lw^k\dx \right|\leq \frac{1}{k}.
		\end{align*}
		In particular, after taking a subsequence, $(\mathcal{B}w^k)$ generates a Young measure, which equals $\nu$ by density of $(f_i)$ and $(g_j)$. We also have that $D^lw^k$ is uniformly bounded in $L^{\infty}(\T^{d})$ and hence $D^lw^k\overset{*}{\rightharpoonup}z$ in $L^{\infty}(\T^{d})$ for some $z$ which has to be zero by the above estimates and density of $(h_n)$. Note that $\supp(\nu_k),\supp(\nu)\subset \overline{B_C(0)}$ uniformly in $k$ for some $C>0$ large enough by the uniform boundedness of the generating sequence. Thus, if we choose $g_1\in C_0(\R^m)$ with $g_1\big|_{B_C(0)}=\id$ in the above sequence, we obtain $\langle\nu,\id\rangle=\lim\limits_{k\rightarrow\infty}\langle\nu_k,g_1\rangle=0$. This shows $\nu\in\mathbb{H}^{q}$.\\
		\\
		\textbf{Case 1: Generating homogeneous Young measures with zero barycenter on the torus.}\\
		Fix $\nu\in\mathcal{P}(\R^m)$ with $\langle \nu,\operatorname{id}\rangle=0$ and $\langle\nu,f\rangle\geq Q^{R}_{\mathcal{B}}f(0)$ for all $f\in C(\R^m)$. Suppose that $\nu\notin\mathbb{H}^{R}$. By the previous claims and the Hahn-Banach Theorem there exist $g\in C_0(\R^m)$ and $\alpha\in\R$ such that
		\begin{align*}
			\langle\mu,g\rangle\geq \alpha\text{ and }\langle\nu,g\rangle <\alpha
		\end{align*}
		for all $\mu\in\mathbb{H}^{R}$. Let $w\in C_c^{\infty}(\mathcal{Q})$ be such that $\|D^lw\|_{L^{\infty}(\T^{d})}\leq R$. Then by Proposition 2.8 in \cite{FM99} the sequence $w_n(x):=\frac{1}{n^l}w(nx)$ satisfies $D^lw_n\overset{*}{\rightharpoonup}0$ in $L^{\infty}(\T^{d})$ and $\|D^lw_n\|_{L^{\infty}(\T^{d})}\leq R$. As $\mathcal{B}w\in L^{\infty}(\T^{d})$ one deduces that $h(\mathcal{B}w_n)=h(\mathcal{B}w)(n\cdot)\overset{*}{\rightharpoonup}\int\limits_{\T^{d}}^{}h(\mathcal{B}w)\dx$ in $L^{\infty}(\T^{d})$ for all $h\in C_0(\R^m)$. Thus,
		\begin{align*}
			\mathcal{B}w_n\overset{Y}{\rightharpoonup}\overline{\delta_{\mathcal{B}w}},
		\end{align*}
		where $\langle\overline{\delta_{\mathcal{B}w}},h\rangle:=\int\limits_{\T^d}^{}h(\mathcal{B}w)\dx$ for all $h\in C_0(\R^m)$. In particular, $\overline{\delta_{\mathcal{B}w}}\in\mathbb{H}^{R}$. Thus,
		\begin{align*}
			\int\limits_{\T^{d}}^{}g(\mathcal{B}w)\dx=\langle\overline{\delta_{\mathcal{B}w}},g\rangle\geq \alpha
		\end{align*}
		which by the definition of $Q^{R}_{\mathcal{B}}g$ implies $Q^{R}_{\mathcal{B}}g(0)\geq \alpha$, a contradiction to $Q^{R}_{\mathcal{B}}g(0)\leq \langle\nu,g\rangle<\alpha$. Therefore, $\nu\in\mathbb{H}^{R}$. So, there exists a sequence $(w_j^{})\subset C_c^{\infty}(\mathcal{Q})$ such that $\mathcal{B}w_j^{}\overset{Y}{\rightharpoonup} \nu$, $D^lw^{}_j\overset{j\rightarrow \infty}{\rightharpoonup}0$ in $L^1(\T^{d})$, and $\underset{j\rightarrow\infty}{\limsup}\|D^lw_j^{}\|_{L^{\infty}(\T^{d})}\leq R$. This finishes the proof for Case 1.\\
		\\
		\textbf{Case 2: Generating inhomogeneous Young measures with zero barycenter on $\Omega$.}\\
		Here we adopt the proof of Proposition 4.4 in \cite{FM99} to our situation.\\
		Define
		\begin{align*}
			\mathbb{X}:=&\{\nu\colon \Omega\to \mathcal{M}(\R^m)\,:\,\nu\text{ is weak*-measurable}, \langle\nu_{x},\operatorname{id}\rangle=0\text{ for a.e.~}x\in\Omega\},\\
			\mathbb{Y}:=&\left\{\nu\in\mathbb{X}\,:\,\text{there exists }(w_j)\subset C_c^{\infty}(\Omega)\text{ s.t.~}\mathcal{B}w_j\overset{Y}{\rightharpoonup}\nu,\ D^lw_j\overset{L^1}{\rightharpoonup}0,\ \underset{j\rightarrow\infty}{\limsup}\|D^lw_j\|_{L^{\infty}(\Omega)}\leq R\right\},\\
			\mathbb{W}:=&\left\{\nu\in\mathbb{X}\,:\,\langle\nu_{x},g\rangle\geq Q^{R}_{\mathcal{B}}g(0)\text{ for all }g\in C_0(\R^m)\text{ and for a.e.~}x\in\Omega\right\}.
		\end{align*}
		We will show that $\mathbb{W}\subset \mathbb{Y}$, which implies statement (ii) for Case 2.\\
		Note that $\overline{\mathbb{Y}}^{\mathcal{M}(\overline{\Omega}\times R^m)}\cap \mathbb{X}=\mathbb{Y}$. To see this, fix $\nu\in \overline{\mathbb{Y}}^{\mathcal{M}(\overline{\Omega}\times R^m)}\cap \mathbb{X}$. Further, let $(f_i)\subset C_c^{\infty}(\Omega)$ be a countable dense subset of $L^1(\Omega)$ and let $(g_j)\subset C_c^{\infty}(\R^m)$ be a countable dense subset of $C_0(\R^m)$. We also choose a countable dense subset $(h_n)$ of $L^1\big(\T^{d},\R^{d^l\times N}\big)$. For a function $\varphi\in C_0(\overline{\Omega}\times \R^m)$ the dual pairing with $\nu$ is defined as $\int\limits_{\Omega}^{}\langle \nu_x,\varphi(x,\cdot)\rangle \dx$. Now as $f_i\cdot g_j\in C_0(\overline{\Omega}\times \R^m)$, we find some $\nu^k\in\mathbb{Y}$ such that
		\begin{align*}
			\left|\int\limits_{\Omega}^{}f_i(x)\left\langle \nu_x-(\nu^k)_x,g_j\right\rangle\dx\right|\leq \frac{1}{k}
		\end{align*}
		for $1\leq i,j\leq k$. On the other hand, for every $k\in\N$ there exists $w^k\in C_c^{\infty}(\Omega)$ such that $\|D^lw^k\|_{L^{\infty}(\Omega)}\leq R+\frac{1}{k}$ and
		\begin{align*}
			\left|\int\limits_{\Omega}^{}f_i(x)\left\langle\big(\nu^k\big)_{x},g_j\right\rangle\dx-\int\limits_{\Omega}^{}f_i(x)g_j(\mathcal{B}w^k(x))\dx \right|&\leq \frac{1}{k},\\
			\left|\int\limits_{\Omega}^{}h_n(x)\cdot D^lw^k(x)\right|&\leq \frac{1}{k}
		\end{align*}
		for $1\leq i,j,n\leq k$. This implies that $D^lw^k\overset{L^1(\Omega)}{\rightharpoonup}0$ and $\nu$ is generated by $(\mathcal{B}w^k)$. Hence, $\nu\in\mathbb{Y}$.\\
		\\
		Now define
		\begin{align*}
			\mathcal{G}_k&:=\left\{\frac{1}{k}(y+\mathcal{Q})\,:\,y\in\Z^d,\ \frac{1}{k}(y+\mathcal{Q})\subset \Omega \right\},\\
			G_k&:=\underset{U\in\mathcal{G}_k}{\bigcup}U,\\
			\mathbb{W}_k&:=\left\{\nu\in\mathbb{W}\,:\,\nu\big|_U\text{ is homogeneous for }U\in\mathcal{G}_k,\ \nu\big|_{(\Omega\backslash G_k)}=\delta_0 \right\},\\
			D&:=\underset{k\in\N}{\bigcup}\mathbb{W}_k,
		\end{align*}
		\textbf{Claim 3:} $\overline{D}^{\mathcal{M}(\overline{\Omega}\times R^m)}\cap \mathbb{W}=\mathbb{W}$.\\
		Let $\nu\in\mathbb{W}$ and define
		\begin{align*}
			\nu^k_{x}:=\begin{cases}
			\displaystyle\underset{U}{\dashint}\nu_{y}\dy,\ & x\in U,\ U\in \mathcal{G}_k,\\
			\delta_0, & \text{ else}.
			\end{cases}
		\end{align*}
		Here
			\begin{align}
				\left\langle\underset{U}{\dashint}\nu_{y}\,\dy,g\right\rangle:=\underset{U}{\dashint}\langle\nu_{y},g\rangle \dy\geq Q^{R}_{\mathcal{B}}g(0)
			\end{align}
			for all $g\in C_0(\R^N)$ and all $U\in\mathcal{G}_k$. It is not difficult to check that this defines a positive linear functional on $C_0(\R^N)$. Hence, the Riesz-Markov representation theorem yields that $\underset{U}{\dashint}\nu_y\,\dy\in \mathcal{P}(\R^N)$. Thus, $\nu^k\in\mathbb{W}_k$. Therefore, we need to show that
		\begin{align*}
			\int\limits_{\Omega}^{}\left\langle(\nu^k)_x,f(x,\cdot)\right\rangle\dx\rightarrow \int\limits_{\Omega}^{}\left\langle\nu_x,f(x,\cdot)\right\rangle\dx
		\end{align*}
		for all $\ f\in C_0(\overline{\Omega}\times \R^m)$. To that end, fix $f\in C_0(\overline{\Omega}\times \R^m)$ and for every $U\in\mathcal{G}_k$ let $x_U\in\left(\frac{1}{k}\Z\right)^d$ denote the corner of $U$ corresponding to the point $0\in [0,1]^d$. Further, define
		\begin{align*}
			\omega(\delta):=\sup\{\|f(x,\cdot)-f(y,\cdot)\|_{C_0(\R^m)}\,:\,x,y\in\overline{\Omega},\ |x-y|\leq \delta \}.
		\end{align*}
		We obtain for sufficiently large $k$
		\begin{align*}
			&\left|\int\limits_{\Omega}^{}\int\limits_{\R^m}^{}f(x,z)\dd\nu_{x}(z)\dx-\int\limits_{\Omega}^{}\int\limits_{\R^m}^{}f(x,z)\dd\nu^k_{x}(z)\dx \right|\\
			\leq&\left|\sum\limits_{U\in\mathcal{G}_k}^{}\int\limits_{U}^{}\int\limits_{\R^m}^{}f(x,z)\dd\nu_{x}(z)\dx-\sum\limits_{U\in\mathcal{G}_k}^{}\int\limits_{U}^{}\underset{U}{\dashint}\int\limits_{\R^m}^{}f(x,z)\dd\nu_{y}(z)\dy\dx \right|+\|f\|_{C_0(\overline{\Omega}\times\R^m)}\left|\Omega \backslash G_k\right|\\
			\leq &2\omega\left(\frac{1}{k} \right)\cdot |\Omega|+\|f\|_{C_0(\overline{\Omega}\times\R^m)}\left|\Omega \backslash G_k\right|,
		\end{align*}
		which tends to zero as $k\rightarrow\infty$.
		\\
		\\
		\textbf{Claim 4:} $\mathbb{W}_k\subset \mathbb{Y}$ for all $k\in\N$.\\
		Fix $k\in\N$. Recall that $\Omega$ is bounded, so we can write $\mathcal{G}_k$ as a finite family of cubes $\mathcal{Q}_i$ with side length $\frac{1}{k}$, i.e.~$\mathcal{G}_k=(\mathcal{Q}_i)_{i=1}^N$ for some $N\in \N$. Let $\nu\in\mathbb{W}_k$ with $\nu\big|_{\mathcal{Q}_i}=\nu^i$. Then by Case 1 above we infer that for all $i=1,\dots,N$ there exist $(w_j^i)_{j\in\N}\subset C_c^{\infty}(\mathcal{Q})$ with $\underset{j\rightarrow\infty}{\limsup}\|D^lw_j^i\|_{L^{\infty}(\T^{d})}\leq R$, $D^lw_j^i\overset{L^1(\T^d)}{\rightharpoonup}0$, and $\mathcal{B}w_j^i\overset{Y}{\rightharpoonup}\nu^i$ as $j\rightarrow\infty$. It is straightforward to check that
		\begin{align*}
			 \tilde{w}_j^i:=k^{-l}w_j^i(k\bullet-x_i)
		\end{align*}
		defined on $\mathcal{Q}_i$ satisfies $\underset{j\rightarrow\infty}{\limsup}\|D^l\tilde{w}_j^i\|_{L^{\infty}(\mathcal{Q}_i)}\leq R$, $D^l\tilde{w}_j^i\overset{L^1(\mathcal{Q}_i)}{\rightharpoonup}0$, and $\mathcal{B}\tilde{w}_j^i\overset{Y}{\rightharpoonup}\nu^i$, where $x_i$ denotes the corner of $\mathcal{Q}_i$ corresponding to the point $0\in [0,1]^d$. Then $\sum\limits_{i=1}^{N}\tilde{w}_j^i\in C_c^{\infty}(\Omega)$ has the desired properties such that $\nu$ lies in $\mathbb{Y}$. This finishes the proof of Claim 4.
		\\
		\\
		Combining the Claims 3 and 4 with the fact that $\overline{\mathbb{Y}}^{\mathcal{M}(\overline{\Omega}\times R^m)}\cap \mathbb{X}=\mathbb{Y}$ yields the desired inclusion $\mathbb{W}\subset \mathbb{Y}$.\\
		\\
		\textbf{Case 3: General Young measures.}\\
		Let $\nu$ be as in (i) in the statement of the theorem. Then the shifted Young measure $x\mapsto\Gamma_{-\langle\nu_{x},\operatorname{id}\rangle}\nu_{x}=:\bar{\nu}_{x}$ lies in $\mathbb{W}$. By the previous discussion this yields a sequence $(u_j)\subset C_c^{\infty}(\Omega)$ with $D^lu_j\overset{L^1(\Omega)}{\rightharpoonup}0$, $\mathcal{B}u_j\overset{Y}{\rightharpoonup}\bar{\nu}$, and $\underset{j\rightarrow\infty}{\limsup}\|D^lu_j\|_{L^{\infty}(\Omega)}\leq R$. Thus, $\mathcal{B}u_j+\langle\nu,\operatorname{id}\rangle\overset{Y}{\rightharpoonup}\nu$, which finishes the proof for sufficiency of (i) for (ii).\\
		\\
		\textbf{Statement (ii) implies (i):}\\
		We apply a localization argument. First, we consider the case where $\nu$ is a Young measure with zero barycenter and is generated by a sequence $(\mathcal{B}u_n)$ with $(u_n)\subset C_c^{\infty}(\Omega)$, $\underset{n\rightarrow\infty}{\limsup}\|D^lu_n\|_{L^{\infty}(\Omega)}\leq R$, and $D^lu_n\overset{L^1(\Omega)}{\rightharpoonup}0$. Let $(f_i)\subset C_c^{\infty}(\mathcal{Q})\subset L^1(\mathcal{Q})$ and $(g_j)\subset C_0(\R^m)$ be countable and dense subsets. We denote by $D_0\subset \Omega$ the set of joint Lebesgue points of the maps $x\mapsto\langle \nu_{x},g_j\rangle$ for all $j\in\N$ in the sense that
		\begin{align*}
			\lim\limits_{r\rightarrow 0}\int\limits_{Q}^{}|\langle \nu_{a+rx},g_j\rangle -\langle \nu_a,g_j\rangle|\dx=0.
		\end{align*}
		Fix $a\in D_0$ and define
		\begin{align*}
			u_{r,n}(x):=\frac{1}{r^l}u_n(a+rx),\ x\in \mathcal{Q},\ r>0\text{ small enough such that }a+r\mathcal{Q}\subset \Omega.
		\end{align*}
		Let $(\varphi_k)\subset C_c^{\infty}(\mathcal{Q})$ be a sequence of smooth cut-off functions such that $\varphi_k\rightarrow 1$ in $L^1(\mathcal{Q})$. Similarly as in Claim 1 we argue that $u_n\rightarrow 0$ in $C^{l-1,\alpha}(\Omega)$ for all $\alpha\in [0,1)$. Hence, after taking a subsequence $(u_{n_k})$ one obtains for all $k$
		\begin{align*}
			\|\varphi_k\|_{C^l(\mathcal{Q})}\|u_{n_k}\|_{C^{l-1}(\Omega)}\rightarrow 0.
		\end{align*}
		Define for all $r>0$ and $k\in \N$ the functions $v_{r,k}\in C_c^{\infty}(\mathcal{Q})$ by
		\begin{align*}
			v_{r,k}(x):=\varphi_k(x)u_{r,n_k}(x).
		\end{align*}
		Then for all fixed $r\leq 1$
		\begin{align}
			\underset{k\rightarrow\infty}{\limsup}\|D^lv_{r,k}\|_{L^{\infty}(\T^{d})}\leq \underset{k\rightarrow\infty}{\limsup}\|D^lu_{n_k}\|_{L^{\infty}(\Omega)}+\underset{k\rightarrow\infty}{\limsup}\,\frac{1}{r^l}\|\varphi_k\|_{C^l(\mathcal{Q})}\|u_{n_k}\|_{C^{l-1}(\Omega)}\leq R.\label{eq:dlconvergence}
		\end{align}
		Moreover, for all $i,j\in\N$ we have
		\begin{align*}
			\lim\limits_{r\rightarrow 0}\lim\limits_{k\rightarrow \infty}\int\limits_{\mathcal{Q}}^{}f_i(x)g_j(\mathcal{B}v_{r,k}(x))\dx&=\lim\limits_{r\rightarrow 0}\lim\limits_{k\rightarrow \infty}\int\limits_{\mathcal{Q}}^{}f_i(x)g_j(\varphi_k(x)(\mathcal{B}u_{n_k})(a+rx))\dx\\
			&=\lim\limits_{r\rightarrow 0}\lim\limits_{k\rightarrow \infty}\underset{a+(0,r)^d}{\dashint}f_i\left(\frac{x-a}{r}\right)g_j((\mathcal{B}u_{n_k})(x))\dx\\
			&=\lim\limits_{r\rightarrow 0}\underset{a+(0,r)^d}{\dashint}f_i\left(\frac{x-a}{r}\right)\langle \nu_{x},g_j\rangle\dx\\
			&=\int\limits_{Q}^{}f_i(x)\dx\langle\nu_{a},g_j\rangle.
		\end{align*}
		Keeping (\ref{eq:dlconvergence}) and the density of $(f_i)$ and $(g_j)$ in mind we choose $r_l=\frac{1}{l}$ for all $l\in\N$ and a diagonal subsequence $w_l:=v_{r_l,k_{r_l}}$ such that
		\begin{align*}
			\underset{l\rightarrow\infty}{\limsup}\|D^lw_l\|_{L^{\infty}(\T^d)}\leq R\text{ and }\mathcal{B}w_l\overset{Y}{\rightharpoonup}\nu_a.
		\end{align*}
		Now if $\nu$ has non-zero barycenter, consider the shifted Young measure $\Gamma_{-\langle\nu,\operatorname{id}\rangle}\nu$ which is generated by the sequence $\mathcal{B}u_j$ satisfying $(u_l)\subset C_c^{\infty}(\Omega)$, $\underset{l\rightarrow\infty}{\limsup}\|D^lu_l\|_{L^{\infty}(\Omega)}\leq R$, and $D^lu_l\overset{L^1(\Omega)}{\rightharpoonup}0$. By the previous case for a.e.~$a\in\Omega$ there is a sequence $(\tilde{w}_l)\subset C_c^{\infty}(\mathcal{Q})$ such that $\underset{l\rightarrow\infty}{\limsup}\|D^l\tilde{w}_l\|_{L^{\infty}(\mathcal{Q})}\leq R$ and $\mathcal{B}\bar{w}_l\overset{Y}{\rightharpoonup} \Gamma_{-\langle\nu_a,\operatorname{id}\rangle}\nu_a$. Hence, we have
		\begin{align*}
			\mathcal{B}\tilde{w}_l+\langle\nu_a,\operatorname{id}\rangle\overset{Y}{\rightharpoonup}\nu_a.
		\end{align*}
		Thus, for $g\in C(\R^m)$ we obtain
		\begin{align}
			\lim\limits_{l\rightarrow \infty}\int\limits_{\T^d}^{}g(\mathcal{B}\tilde{w}_l(x)+\langle\nu_a,\operatorname{id}\rangle)\dx=\langle\nu_a,g\rangle.\label{eq:localizationconvergence}
		\end{align}
		Note that here one needs to replace $g$ by some function in $C_0(\R^m)$, which coincides with $g$ on a ball containing the uniformly bounded functions $\mathcal{B}\tilde{w}_l(x)+\langle\nu_a,\operatorname{id}\rangle$.\\
		On the other hand, for all $\varepsilon>0$ the definition of $Q_{\mathcal{B}}^{R+\varepsilon}g$ implies
		\begin{align}
			\int\limits_{\T^d}^{}g(\mathcal{B}\tilde{w}_l(x)+\langle\nu_a,\operatorname{id}\rangle)\dx\geq Q_{\mathcal{B}}^{R+\varepsilon}g(\langle\nu_a,\operatorname{id}\rangle)\label{eq:jensentype}
		\end{align}
		for all $g\in C(\R^m)$ if $l$ is large enough. Combining (\ref{eq:localizationconvergence}) and (\ref{eq:jensentype}) and taking the limits over $l$ and $\varepsilon$ finishes the proof by using Lemma \ref{lem:qlimit}.	
	\end{proof}
\section{Weak and Measure-Valued Solutions}\label{sect:weakandmvs}
Let $\Omega\subset \R^d$ be open or the torus $\T^d$ and let $T>0$. A pair $(\rho,m)\in L^{\infty}((0,T)\times \Omega,\R^+\times \R^d)$ is called \textit{weak solution} of (\ref{eq:euler}) with initial data $(\rho_0,m_0)\in L^{\infty}(\Omega)$ if
\begin{align*}
	\int\limits_{0}^{T}\int\limits_{\Omega}^{}\partial_t\varphi\cdot m+\nabla_x\varphi:\frac{m\otimes m}{\rho}+\rho^{\gamma}\operatorname{div}_x\varphi\dx\dt+\int\limits_{\Omega}^{}\varphi(0,\cdot)\cdot m_0\dx&=0,\\
	\int\limits_{0}^{T}\int\limits_{\Omega}^{}\rho\partial_t\psi+\nabla_x\psi\cdot m\dx\dt+\int\limits_{\Omega}^{}\rho_0\psi(0,\cdot)\dx&=0
\end{align*}
for all $\varphi\in C_c^{\infty}([0,T)\times \Omega,\R^d)$ and $\psi\in C_c^{\infty}([0,T)\times \Omega)$, where the above integrals have to exist as part of the definition.\\
Define the function
\begin{align*}
	e\colon \R^+\times \R^d\to \R,\ (\rho,m)\mapsto \frac{|m|^2}{2\rho}+\frac{1}{\gamma-1}\rho^{\gamma}.
\end{align*}
A weak solution $(\rho,m)$ with initial data $(\rho_0,m_0)$ is called \textit{admissible weak solution} to (\ref{eq:euler}) if it satisfies
\begin{align}
	\int\limits_{\Omega}^{}e(\rho,m)(t,x)\dx\leq \int\limits_{\Omega}^{}e(\rho_0,m_0)\dx\label{eq:energyinequalitydefinition}
\end{align}
for a.e.~$t\in (0,T)$.\\
\\
Next, let us introduce an even weaker notion of solution:\\
A Young measure $\nu\in L^{\infty}_{\operatorname{w}}((0,T)\times \Omega,\mathcal{P}(\R^+\times \R^m))$ is called a \textit{measure-valued solution} to (\ref{eq:euler}) with initial data $(\rho_0,m_0)\in L^{\infty}(\Omega)$ if
\begin{align*}
	\int\limits_{0}^{T}\int\limits_{\Omega}^{}\partial_t\varphi\cdot \langle\nu,m\rangle+\nabla_x\varphi:\left\langle\nu,\frac{m\otimes m}{\rho}\right\rangle+\left\langle\nu,\rho^{\gamma}\right\rangle\operatorname{div}_x\varphi\dx\dt+\int\limits_{\Omega}^{}\varphi(0,\cdot)\cdot m_0\dx&=0,\\
	\int\limits_{0}^{T}\int\limits_{\Omega}^{}\langle\nu,\rho\rangle\partial_t\psi+\nabla_x\psi\cdot \langle\nu,m\rangle\dx\dt+\int\limits_{\Omega}^{}\rho_0\psi(0,\cdot)\dx&=0
\end{align*}
for all $\varphi\in C_c^{\infty}([0,T)\times \Omega,\R^d)$ and $\psi\in C_c^{\infty}([0,T)\times \Omega)$, where the above integrals have to exist as part of the definition. Also the notation
\begin{align*}
	\langle\nu,f\rangle:=&\int\limits_{\R^+\times \R^d}^{}f(\xi)\dd\nu(\xi)
\end{align*}
for any continuous function $f\in C_0(\R^m)$ and $m\in\N$ has been used.\\
We call a measure-valued solution $\nu$ with initial data $(\rho_0,m_0)$ an \textit{admissible measure-valued solution} if
\begin{align*}
	\int\limits_{\Omega}^{}\langle\nu_{(t,x)},e\rangle\dx\leq \int\limits_{\Omega}^{}e(\rho_0,m_0)\dx
\end{align*}
for a.e.~$t\in(0,T)$.\\
\\
The proof of our main result uses a detour over a linear relaxation of the Euler equations which then can be handled by more abstract arguments involving homogeneous differential operators. The relaxed Euler equations are given by
\begin{equation}
	\begin{aligned}
		\partial_t m+\operatorname{div}_xM+\nabla_xQ&=0,\\
		\partial_t\rho+\operatorname{div}_xm&=0\label{eq:linearizedeuler}
	\end{aligned}
\end{equation}
which is a linear first order PDE in the variables $(\rho,m,M,Q)$ with values in $\R^+\times \R^d\times \mathcal{S}_0^d\times \R^+$. We will call a quadruple $(\rho,m,M,Q)\in L^{\infty}((0,T)\times \T^d)$ a \textit{weak subsolution} with corresponding initial data $(\rho_0,m_0)\in L^{\infty}((0,T)\times \T^d)$ if it solves (\ref{eq:linearizedeuler}) tested by smooth functions with compact support in $[0,T)\times \T^d$.\\
Further, define the lifting map $\Theta\colon [0,\infty)\times \R^d\to \R^m$ by
\begin{align*}
	\Theta(\rho,m)=\left(\rho,m,\frac{m\ocircle m}{\rho},\rho^{\gamma}+\frac{|m|^2}{d\rho} \right).
\end{align*}
The lift $\tilde{\nu}$ of a Young measure $\nu\colon (0,T)\times \T^d\to \mathcal{P}(\R^m)$ is defined for all $f\in C_0(\R^m)$ by
\begin{align*}
	\langle \tilde{\nu}_{(t,x)},f\rangle = \langle \nu_{(t,x)},f\circ \Theta\rangle.
\end{align*}
It is straightforward to check that $\nu$ is a measure-valued solution if and only if the barycenter of its lift $\langle\tilde{\nu},\operatorname{id}\rangle$ is a weak solution of~\eqref{eq:linearizedeuler}.\\
In Section \ref{sect:necessaryconditions} we will encounter the notion of a \textit{vanishing viscosity sequence}. By this we mean a sequence $(\rho_n,m_n)\subset L^{\infty}((0,T)\times \Omega)$ of distributional solutions to the compressible Navier-Stokes equations
\begin{align*}
	\partial_tm_n+\operatorname{div}\left(\frac{m_n\otimes m_n}{\rho_n} \right)+\nabla \rho_n^{\gamma}&=\mu_n\operatorname{div}\mathbb{S}\left(\nabla \frac{m_n}{\rho_n}\right),\\
	\partial_t\rho_n+\operatorname{div}m_n&=0
\end{align*}
with initial data $(\rho_n^0,m_n^0)\in L^{\infty}(\Omega)$. Here, $(\mu_n)\subset \R^+$ is a sequence converging to zero and $\mathbb{S}$ denotes the Newtonian viscous stress tensor given by
\begin{equation*}
	\mathbb S(A)=\mu\left(A+A^t-\frac23 \operatorname{tr}A \mathbb E_d\right)+\eta \operatorname{tr}A \mathbb E_d,
\end{equation*}
with constants $\mu>0$, $\eta\geq 0$. If $(\rho_n,m_n)$ satisfies the energy inequality (\ref{eq:energyinequalitydefinition}) for a.e.~$t\in (0,T)$, we call this sequence \textit{admissible}.\\
\\
The following lemma collects some properties of admissible measure-valued solutions.
\begin{lem}\label{lem:initialvalueproperties}
	Let $\nu$ be an admissible measure-valued solution on $(0,T)\times \Omega$ with initial data $(\rho_0,m_0)\in L^{\infty}(\Omega)$. Assume $\nu$ satisfies
	\begin{align*}
		\supp\nu&\subset \{(\rho,m)\in \R^+\times \R^d\,:\, \rho\geq \eta \},\\
		\langle\nu,\rho\rangle&\leq R,\\
		\langle\nu,m\rangle&\leq R
	\end{align*}
	a.e.~on $(0,T)\times \Omega$ for some $0<\eta\leq R<\infty$.\\
	Then $\nu$ has the following properties:
	\begin{align*}
		\langle\nu,\rho\rangle&\in C([0,T],L_{\operatorname{w}}^{\gamma}(\Omega)),\\
		\langle\nu,m\rangle&\in C([0,T],L_{\operatorname{w}}^{2}(\Omega)),\\
		\left\langle\nu_{(t,\cdot)},\rho\right\rangle&\rightarrow \left\langle\nu_{(0,\cdot)},\rho\right\rangle=\rho_0\text{ in }L^{\gamma}(\Omega)\text{ as }t\rightarrow 0,\\
		\left\langle\nu_{(t,\cdot)},m\right\rangle&\rightarrow \left\langle\nu_{(0,\cdot)},m\right\rangle=m_0\text{ in }L^{2}(\Omega)\text{ as }t\rightarrow 0,\\
		\int\limits_{\Omega}^{}\left\langle\nu_{(t,x)},e\right\rangle\dx&\rightarrow\int\limits_{\Omega}^{}e(\rho_0,m_0)\dx\text{ as }t\rightarrow 0,\\
		\langle\nu,\rho\rangle&=\langle\tilde{\nu},\rho\rangle,\\
		\langle\nu,m\rangle&=\langle\tilde{\nu},m\rangle.
	\end{align*}
\end{lem}
\begin{proof}
	As $\operatorname{pr}_{\rho}\circ\Theta=\operatorname{pr}_{\rho}$ and $\operatorname{pr}_{m}\circ\Theta=\operatorname{pr}_{m}$, it is immediate that $\langle\nu,\rho\rangle=\langle\tilde{\nu},\rho\rangle$ and $\langle\nu,m\rangle=\langle\tilde{\nu},m\rangle$.\\
	Proposition \ref{prop:appendixa} below yields $\langle\nu,\rho\rangle\in C([0,T],L^{\gamma}_{\operatorname{w}}(\Omega))$ and $\langle\nu,m\rangle\in C([0,T],L^{2}_{\operatorname{w}}(\Omega))$. Hence, the definition of measure-valued solutions implies $\left\langle\nu_{(0,\cdot)},\rho\right\rangle=\rho_0$ and $\left\langle\nu_{(0,\cdot)},m\right\rangle=m_0$. Thus,
	\begin{align*}
		\liminf\limits_{t\rightarrow 0}\|\langle\nu_{(t,\cdot)},\rho\rangle\|_{L^{\gamma}(\Omega)}&\geq \|\rho_0\|_{L^{\gamma}(\Omega)},\\
		\liminf\limits_{t\rightarrow 0}\|\langle\nu_{(t,\cdot)},m\rangle\|_{L^{2}(\Omega)}&\geq \|m_0\|_{L^{2}(\Omega)}.
	\end{align*}
	Note that the function $(\rho,m)\mapsto \frac{|m|^2}{2\rho}$ is convex on $(0,\infty)\times \R^d$, hence so is the function $e$. Therefore, we can estimate using Jensen's inequality
	\begin{align}
		\underset{t\rightarrow 0}{\limsup}\int\limits_{\Omega}^{}\frac{|\langle\nu,m\rangle|^2}{2\langle\nu,\rho\rangle}(t,x)\dx\leq \underset{t\rightarrow 0}{\limsup}\int\limits_{\Omega}^{}\langle\nu,e\rangle(t,x)\dx-\underset{t\rightarrow 0}{\liminf}\int\limits_{\Omega}^{}\frac{1}{\gamma-1}\langle\nu,\rho\rangle^{\gamma}(t,x)\dx\leq \int\limits_{\Omega}^{}\frac{|m_0|^2}{2\rho_0}\dx.\label{eq:initialvalueestimate}
	\end{align}
	Observe that the set
	\begin{align*}
		\mathcal{K}:=\{(\rho,m)\in (L^{\gamma}\times L^2)(\Omega)\,:\,\eta\leq \rho\leq R\text{ and }m\leq R\text{ a.e.~on }\Omega \}
	\end{align*}
	is convex and strongly closed.\\
	We now show that the functional $F\colon \mathcal{K}\to\R,\ (\rho,m)\mapsto \int\limits_{\Omega}^{}\frac{|m|^2}{2\rho}\dx$ is strongly continuous:\\
	To this end, let $(\rho_j,m_j)\subset \mathcal{K}$ such that $\rho_j\rightarrow \rho$ in $L^{\gamma}(\Omega)$ and $m_j\rightarrow m$ in $L^2(\Omega)$. Then
	\begin{align*}
		\left|\int\limits_{\Omega}^{}\frac{|m_j|^2}{2\rho_j}-\frac{|m|^2}{2\rho}\dx \right|\leq \frac{R}{\eta}|\Omega|^{\frac{1}{2}}\|m_j-m\|_{L^2(\Omega)}+\frac{R^2}{2}|\Omega|^{1-\frac{1}{\gamma}}\cdot \frac{1}{\eta^2}\|\rho_j-\rho\|_{L^{\gamma}(\Omega)}\rightarrow 0.
	\end{align*}
	By a Hahn-Banach argument we infer that $F$ is a weakly lower semicontinuous functional on $\mathcal{K}$. Therefore, as $(\langle\nu,\rho\rangle(t,\cdot),\langle\nu,m\rangle(t,\cdot))\in\mathcal{K}$ for every $t\in [0,T]$, we obtain for $t\rightarrow 0$ that
	\begin{align*}
		\int\limits_{\Omega}^{}\frac{|\langle\nu,m\rangle|^2}{2\langle\nu,\rho\rangle}(t,x)\dx\rightarrow \int\limits_{\Omega}^{}\frac{|m_0|^2}{2\rho_0}\dx.
	\end{align*}
	This also implies by an analogous estimate as in (\ref{eq:initialvalueestimate}) that $\langle\nu,\rho\rangle(t,\cdot)\rightarrow \rho_0$ in $L^{\gamma}(\Omega)$.\\
	Now let $\varphi\in C_c^{\infty}(\Omega)$ and observe that
		\begin{align*}
			&\left|\int\limits_{\Omega}^{}\left(\frac{\langle\nu,m\rangle}{\sqrt{\langle\nu,\rho\rangle}}(t,x)-\frac{m_0}{\sqrt{\rho_0}}(x) \right)\varphi(x)\dx\right|\\
			\leq &\left|\int\limits_{\Omega}^{}\left(\langle\nu,m\rangle(t,x)-m_0(x) \right)\frac{\varphi}{\sqrt{\rho_0}}(x)\dx\right|\\
			&+\left|\int\limits_{\Omega}^{}\langle\nu,m\rangle(t,x)\left(\frac{1}{\sqrt{\langle\nu,\rho\rangle}}(t,x)-\frac{1}{\sqrt{\rho_0}}(x) \right)\varphi(x)\dx\right|\\
			\leq &\left|\int\limits_{\Omega}^{}\left(\langle\nu,m\rangle(t,x)-m_0(x) \right)\frac{\varphi}{\sqrt{\rho_0}}(x)\dx\right|+R\|\varphi\|_{L^{\infty}}\int\limits_{\Omega}^{}\frac{|\sqrt{\rho_0}(x)-\sqrt{\langle\nu,\rho\rangle}(t,x)|}{\eta}\dx\\
			\leq &\left|\int\limits_{\Omega}^{}\left(\langle\nu,m\rangle(t,x)-m_0(x) \right)\frac{\varphi}{\sqrt{\rho_0}}(x)\dx\right|+R\|\varphi\|_{L^{\infty}}\int\limits_{\Omega}^{}\frac{|\rho_0(x)-\langle\nu,\rho\rangle(t,x)|}{2\eta^{\frac{3}{2}}}\dx\rightarrow 0,
		\end{align*}
		where we used the weak convergence $\langle\nu,m\rangle(t,\cdot)\rightharpoonup m_0$. So, we have shown that (after using a standard approximation argument for the test functions) that $\frac{\langle\nu,m\rangle}{\sqrt{\langle\nu,\rho\rangle}}(t,\cdot)\rightharpoonup \frac{m_0}{\sqrt{\rho_0}}$ in $L^2(\Omega)$. Together with the already proven convergence of the $L^2$-norms this yields that $\frac{\langle\nu,m\rangle}{\sqrt{\langle\nu,\rho\rangle}}(t,\cdot)\rightarrow \frac{m_0}{\sqrt{\rho_0}}$ in $L^2(\Omega)$. From that, we can infer that $\langle\nu,m\rangle(t,\cdot)\rightarrow m_0$ in $L^2(\Omega)$. Indeed, due to the already known weak convergence, it suffices to check that the $L^2$-norms converge. We estimate
		\begin{align*}
			&\left|\int\limits_{\Omega}^{}|\langle\nu,m\rangle|^2(t,x)-|m_0|^2(x)\dx \right|\\
			\leq &\left|\int\limits_{\Omega}^{}\langle\nu,\rho\rangle(t,x)\left(\frac{|\langle\nu,m\rangle|^2}{\langle\nu,\rho\rangle}(t,x)-\frac{|m_0|^2}{\rho_0}(x)\right)\dx \right|+\left|\int\limits_{\Omega}^{}\frac{|m_0|^2}{\rho_0}(x)\left(\langle\nu,\rho\rangle(t,x)-\rho_0(x)\right)\dx \right|\\
			\leq &\frac{2R^2}{\sqrt{\eta}}\int\limits_{\Omega}^{}\left|\frac{\langle\nu,m\rangle}{\sqrt{\langle\nu,\rho\rangle}}(t,x)-\frac{m_0}{\sqrt{\rho_0}}(x)\right|\dx+\frac{R^2}{\eta}\int\limits_{\Omega}^{}|\langle\nu,\rho\rangle(t,x)-\rho_0(x)|\dx.
		\end{align*}
		The right hand side now goes to zero as $t\rightarrow 0^+$ by the boundedness of $\Omega$ and the already shown strong convergence of $\langle\nu,\rho\rangle(t,\cdot)$ and $\frac{\langle\nu,m\rangle}{\sqrt{\langle\nu,\rho\rangle}}(t,\cdot)$.\\
	For the convergence involving $e$ we obtain from Jensen's inequality that
	\begin{align*}
		\underset{t\rightarrow 0}{\liminf}\int\limits_{\Omega}^{}\langle\nu,e\rangle(t,x)\dx&\geq\lim\limits_{t\rightarrow 0}\int\limits_{\Omega}^{}\frac{|\langle\nu,m\rangle|^2}{2\langle\nu,\rho\rangle}(t,x)\dx+\lim\limits_{t\rightarrow 0}\int\limits_{\Omega}^{}\frac{1}{\gamma-1}\langle\nu,\rho\rangle^{\gamma}(t,x)\dx\\
		&=\int\limits_{\Omega}^{}e(\rho_0,m_0)\dx.
	\end{align*}
	Combining this with the admissibility of $\nu$ finishes the proof.
\end{proof}
We used the following proposition in the previous proof. It is a version for $L^p$ of a result shown in Appendix A of \cite{DLSz10}. The proof follows analogously to that for $L^2$ and hence will be omitted.
\begin{prop}\label{prop:appendixa}
	Let $1<p<\infty$ and $\Omega\subset \R^d$ open. If $a\in L^{\infty}((0,T),L^p(\Omega))$ and $b,c\in L^1_{\operatorname{loc}}((0,T)\times \Omega)$ solve
	\begin{align*}
		\partial_t a+\operatorname{div}_x b+\nabla_x c=0.
	\end{align*}
	Then, after redefining on a set of measure zero, it holds that $a\in C([0,T],L^p_{\operatorname{w}}(\Omega))$.
\end{prop}
\section{Proof of the main results}\label{sect:proofofthemainresults}
	In this section we prove the two main results Theorem \ref{theo:weakmainresult} and Theorem \ref{theo:mainresult}. The proof will be split into several smaller steps which will be formulated as individual propositions.\\
	A crucial point in the proof is to consider the relaxed Euler equations (\ref{eq:linearizedeuler}) and use Theorem \ref{theo:fromsubsolntosequence}. For that note that the system (\ref{eq:linearizedeuler}) can be rephrased as a linear homogeneous differential operator
	\begin{align*}
		\mathcal{A}_E\colon C^{\infty}\left(\R^{d+1},\R^{\left(1+\frac{d}{2}\right)(d+1)}\right)\to C^{\infty}\big(\R^{d+1},\R^{d+1}\big)
	\end{align*}
	of order one. Actually, here we will work on the level of potentials. In Proposition 4.1 in \cite{G20} an operator
	\begin{align*}
		\mathcal{B}_E\colon C^{\infty}\left(\R^{d+1},\R^{\frac{1}{4}(d+1)^2d^2}\right)\to C^{\infty}\left(\R^{d+1},\R^{\left(1+\frac{d}{2}\right)(d+1)}\right)
	\end{align*}
	is constructed such that for all $z\in C^{\infty}(\R^{d+1})$ it holds that $\mathcal{A}_Ez=0$ if and only if there exists some $u\in C^{\infty}(\R^{d+1})$ with $z=\mathcal{B}_Eu$. It is not immediate that $\mathcal{B}_E$ is a potential in the sense of (\ref{eq:potentialdefinition}). Nevertheless, in the case $d=2$ we will show in Section \ref{sect:application} that, indeed, $\mathcal{B}_E$ is a potential for $\mathcal{A}_E$ in that sense. For the present section the only important property of $\mathcal{B}_E$ is that it is homogeneous of order two.\\
	Recall from (\ref{eq:gammadefinition}) that we only consider the case $\gamma=1+\frac{2}{d}$.
	Define the map
	\begin{align*}
		e_{\operatorname{kin}}(\rho,m,M):=\frac{d}{2}\lambda_{\operatorname{max}}\left(\frac{m\otimes m}{\rho}-M \right).
	\end{align*}
	The notation above indicating the \textit{kinetic energy density} is taken from \cite{EJT20} which itself is adopted from \cite{DLSz10}. The following theorem is a consequence of the genuine compressible convex integration method developed in \cite{EJT20}, which is a generalization of the results in \cite{M20} to non-constant energies. This will be used as a black-box for convex integration in the proof of Proposition \ref{prop:stepthree}. We only state this result on the torus which will be our use case. Note that the assumption of our space-dimension $d$ being greater than one is important here, since convex integration is not feasible in only one space-dimension.
	\begin{theo}\label{theo:EJT}
		Let $d\geq 2$. Let $(\rho',m',M',Q')\in C((0,T)\times \T^d)$ be a subsolution in $(0,T)\times \T^d$ such that $(\rho',m')\in C([0,T],(L^{\gamma}\times L^2)_{\operatorname{w}}(\T^d))$ and $Q'>p(\rho')+\frac{2}{d}e_{\operatorname{kin}}(\rho',m',M')$
		in $(0,T)\times \T^d$. Then there exist infinitely many weak solutions $(\rho,m)\in C([0,T],(L^{\gamma}\times L^2)_{\operatorname{w}}(\T^d))$ of (\ref{eq:euler}) such that
		\begin{align*}
			\rho&=\rho'\text{ on }(0,T)\times \T^d,\\
			Q'&=\rho'^{\gamma}+\frac{|m|^2}{d\rho'}\text{ almost everywhere on }(0,T)\times \T^d,\\
			m(0,\cdot)&=m'(0,\cdot).
		\end{align*}
	\end{theo}
	\begin{proof}
		This follows from Theorem 4.2 in \cite{EJT20}.
	\end{proof}
	In the next lemma we collect some properties of the kinetic energy density.
	\begin{lem}\label{lem:lemmatwopointtwo}
		It holds that:
		\begin{itemize}
			\item The map $e_{\operatorname{kin}}\colon \R^+\times \R^d \times \mathcal{S}_0^d\to \R$ is convex.
			\item The inequality $\frac{|m|^2}{2\rho}\leq e_{\operatorname{kin}}(\rho,m,M)$ holds, with equality if and only if $M=\frac{m\otimes m}{\rho}-\frac{|m|^2}{d\rho}\mathbb{E}_d$.
			\item For all $(\rho,m,M)\in \R^+\times \R^d\times \in S_0^d$ we have $|M|\leq \frac{2(d-1)}{d}e_{\operatorname{kin}}(\rho,m,M)$.
		\end{itemize}
	\end{lem}
	\begin{proof}
		See Lemma 2.2 in \cite{EJT20}.
	\end{proof}
	Let us now discuss the first step towards proving our main results Theorem \ref{theo:weakmainresult} and Theorem \ref{theo:mainresult}.
	\begin{prop}\label{theo:stepone}
		Let $T>0$. Let $\nu$ be an admissible measure-valued solution of (\ref{eq:euler}) with initial data $(\rho_0,m_0)\in L^{\infty}(\T^d)$. Suppose $\nu$ fulfills the following conditions:
		\begin{itemize}
			\item There exists $\eta>0$ such that
			\begin{align*}
				\supp\left(\tilde{\nu}_{(t,x)}\right)\subset\left\{(\rho,m,M,Q)\,:\, \rho\geq \eta\text{ and }Q=\frac{2}{d}\langle\nu_{(t,x)},e\rangle\right\}
			\end{align*}
			for a.e.~$(t,x)\in (0,T)\times\T^d$.
			\item The map $(t,x)\mapsto \langle\nu_{(t,x)},e\rangle$ is continuous on $[0,T]\times \T^d$.
			\item There exists $R>0$ such that the lift $\tilde{\nu}$ satisfies $\langle\tilde{\nu}_{(t,x)},f\rangle\geq Q^{R}_{\mathcal{B}_E}f(\langle\tilde{\nu}_{(t,x)},\operatorname{id}\rangle)$ for a.e.~$(t,x)\in (0,T)\times\T^d$ and all $f$ continuous.
			\item The lift $\tilde{\nu}$ satisfies $\langle\tilde{\nu},\operatorname{id}\rangle=\sigma+\mathcal{B}_Ew$ for some $w\in W^{2,\infty}((0,T)\times \T^d)$ and some $\sigma\in C([0,T]\times\T^d)$.
		\end{itemize}
		Then the lifted measure $\tilde{\nu}$ is generated by a uniformly bounded sequence of subsolutions $(\rho_j,m_j,M_j,Q_j)\in C([0,T]\times \T^d)\cap C^{\infty}((0,T)\times \T^d)$ with the property that:
		\begin{align*}
			\rho_j&\geq \frac{\eta}{4}\text{ on }(0,T)\times \T^d,\\
			\frac{2}{d}e_{\operatorname{kin}}(\rho_j,m_j,M_j)+\rho_j^{\gamma}&< Q_j\text{ on }(0,T)\times 
			\T^d,\\
			\left|Q_j(t,x)-\frac{2}{d}\langle\nu_{(t,x)},e\rangle\right|&\leq \frac{1}{j}\text{ for all }(t,x)\in[0,T]\times \T^d,\\
			\rho_j(0,\cdot)&\rightarrow \rho_0\text{ in }L^{\gamma}(\T^d),\\
			m_j(0,\cdot)&\rightarrow m_0\text{ in }L^{2}(\T^d).
		\end{align*}
	\end{prop}
	Recall that we introduced the total energy density $e\colon (\rho,m)\mapsto\frac{1}{\gamma-1}\rho^{\gamma}+\frac{|m|^2}{2\rho}$.
	\begin{proof}
		From Theorem \ref{theo:fromsubsolntosequence} we obtain a sequence $(u_j)\subset C_c^{\infty}((0,T)\times\T^d)$ such that
		\begin{align*}
			\underset{j\rightarrow\infty}{\limsup}\|D^2u_j\|_{L^{\infty}((0,T)\times\T^d)}\leq R,\ \ \ \  D^2u_j\overset{L^1}{\rightharpoonup}0,\ \ \ \ \mathcal{B}_Eu_j+\langle \tilde{\nu},\id\rangle\overset{Y}{\rightharpoonup}\tilde{\nu}.
		\end{align*}
		Since $\langle\tilde{\nu},\id\rangle=\sigma+\mathcal{B}_Ew$, we consider the sequence $(u_j+w)\subset W^{2,\infty}((0,T)\times\T^d)$. This sequence satisfies
		\begin{align}
			\|D^2(u_j+w)\|_{L^{\infty}((0,T)\times \T^d)}\leq R',\ \ \ \ D^2(u_j+w)\overset{L^1}{\rightharpoonup}D^2w,\ \ \ \ \mathcal{B}_E(u_j+w)\overset{Y}{\rightharpoonup}\tau_{-\sigma}\tilde{\nu}\label{eq:arxivresultconditions}
		\end{align}
		for some $R'>0$. We used Lemma \ref{lem:measureproperties}, where also the notion of the shifted Young measure $\tau_{-\sigma}\tilde{\nu}$ has been defined. In particular, $\|\mathcal{B}_E(u_j+w)\|_{L^{\infty}((0,T)\times\T^d)}\leq \tilde{R}$ for some $\tilde{R}>0$. Define $\bar{R}:=\tilde{R}+\|\sigma\|_{L^{\infty}}+\eta^{\gamma}-\left(\frac{\eta}{2} \right)^{\gamma}$ and define the set
		\begin{align*}
			K_{(t,x)}:=&\left\{(\rho,m,M,Q) \,:\, \frac{2}{d}e_{\operatorname{kin}}(\rho,m,M)+\rho^{\gamma}\leq Q\text{ and }Q= \frac{2}{d}\left\langle\nu_{(t,x)},e\right\rangle \right\}\\
			&\cap \left(\left[\frac{\eta}{2},\infty\right)\times \R^d\times S_0^d\times\R^+\right) \cap [-\bar{R},\bar{R}]^N,
		\end{align*}
		where $N:=\left(1+\frac{d}{2} \right)\cdot (d+1)$. The definition of the lifted measure together with Lemma \ref{lem:lemmatwopointtwo} and the Fundamental Theorem of Young Measure Theory, cf.~for example \cite{FM99} Theorem 2.2, imply that $\supp\left(\tilde{\nu}_{(t,x)} \right)\subset K_{(t,x)}$ for a.e.~$(t,x)\in (0,T)\times \T^d$, since its generating sequence $\sigma+\mathcal{B}_E(u_j+w)$ is uniformly bounded by $\tilde{R}+\|\sigma\|_{L^{\infty}([0,T]\times \T^d)}$.\\
		\\
		\textbf{Claim:} $K_{(t,x)}$ is non-empty, convex, compact for all $(t,x)\in [0,T]\times \T^d$, and for every $\varepsilon>0$ there exists some $\delta>0$ such that
		\begin{align}
			d_H\left(K_{(t,x)},K_{(t',x')} \right)\leq \varepsilon \text{ for all }(t,x),(t',x')\in (0,T)\times \T^d\text{ with }|(t,x)-(t',x')|\leq \delta,\label{eq:hdcontinuity}
		\end{align}
		where $d_H$ denotes the Hausdorff-distance.\\
		\\
		Indeed, $K_{(t,x)}\subset \R^m$ is clearly bounded and closed, hence compact. Since $(\rho,m,M,Q)\mapsto \frac{2}{d}e_{\operatorname{kin}}(\rho,m,M)+\rho^{\gamma}-Q$ is convex by Lemma \ref{lem:lemmatwopointtwo}, we infer that $K_{(t,x)}$ is the intersection of convex sets.\\
		By assumption and the definition of the lift $\Theta$ we have for every $(\rho,m,M,Q)\in \supp\left(\tilde{\nu}_{(t,x)}\right)$ that $Q=\frac{|m|^2}{d\rho}+\rho^{\gamma}=\frac{2}{d}\left\langle\nu_{(t,x)},e\right\rangle$ and $\rho\geq \eta$ for a.e.~$(t,x)\in(0,T)\times \T^d$. Continuity of $(t,x)\mapsto \left\langle\nu_{(t,x)},e\right\rangle$ implies that
		\begin{align*}
			\frac{2}{d}\left\langle\nu_{(t,x)},e\right\rangle\geq \rho^{\gamma}\geq \eta^{\gamma}\geq \left(\frac{\eta}{2} \right)^{\gamma}\text{ on }[0,T]\times \T^d.
		\end{align*}
		Thus, for all $(t,x)\in[0,T]\times \T^d$ and for all $\rho\in \left[\frac{\eta}{2},\left(\frac{2}{d}\left\langle\nu_{(t,x)},e\right\rangle \right)^{\frac{1}{\gamma}} \right]$, we can infer that $\left(\rho,0,0,\frac{2}{d}\left\langle\nu_{(t,x)},e\right\rangle \right)\in K_{(t,x)}$. In particular, $K_{(t,x)}$ is nonempty. Hence, the uniform continuity with respect to the Hausdorff-distance follows from Lemma \ref{lem:hddistance} below. This proves the claim.\\
		\\
		In order to improve the sequence $(u_j+w)$ we have to ensure that the sets $(K_{(t,x)}-\sigma(t,x))$ behave sufficiently well in $(t,x)$. Observe that $\supp(\tau_{-\sigma(t,x)}\tilde{\nu}_{(t,x)})\subset K_{(t,x)}-\sigma(t,x)$ and $K_{(t,x)}-\sigma(t,x)$ is again compact and convex for all $(t,x)$. Note also that the map $(t,x)\mapsto K_{(t,x)}-\sigma(t,x)$ is uniformly continuous with respect to the Hausdorff-distance on $(0,T)\times \T^d$. In the proof of the claim above we determined explicit elements of $K_{(t,x)}$. Since the $\rho$-coordinate of these elements satisfies $\rho\in \left[\frac{\eta}{2},\left(\frac{2}{d}\left\langle\nu_{(t,x)},e\right\rangle \right)^{\frac{1}{\gamma}} \right]\supset \left[\frac{\eta}{2},\eta \right]$, we obtain that $|K_{(t,x)}-\sigma(t,x)|_{\infty}\geq \frac{\eta}{4}$ for all $(t,x)\in (0,T)\times \T^d)$.\\
		Combining this with the claim and (\ref{eq:arxivresultconditions}) we can apply Corollaries 2.3, 2.4, and 2.5 from \cite{G20}. This yields a sequence of open sets $(U_j)$ and a sequence of functions $(g_j)\subset W^{2,\infty}((0,T)\times \T^d)$ such that
		\begin{align*}
			U_j\subset \subset (0,T)\times &(0,1)^{d},\\
			\left|\big((0,T)\times \T^d\big)\backslash U_j\right|\rightarrow &0,\\
			g_j=w\text{ on }\big((0,T)\times \T^d \big)\backslash &U_j,\\
			\left\|(t,x)\mapsto\dist\left(\mathcal{B}_Eg_j(t,x),K_{(t,x)}-\sigma(t,x)\right)\right\|_{L^{\infty}((0,T)\times \T^d)}\rightarrow &0,\\
			\mathcal{B}_Eg_j\overset{Y}{\rightharpoonup}&\tau_{-\sigma}\tilde{\nu}.
		\end{align*}		
		We need to mollify this sequence of functions to be able to apply Proposition \ref{theo:EJT} later on. We follow Step 4 of the proof of Theorem 18 in \cite{SW12}. So, let $\psi\colon \R^d\to\R$ be a standard mollifier with support on $B_1(0)$ and let $\chi\colon \R\to\R$ be a shifted standard mollifier such that its support lies in $(-1,0)$. As usual define for $\varepsilon>0$ the functions $\psi_{\varepsilon}(x):=\frac{1}{\eps^d}\psi\left(\frac{x}{\eps} \right)$ and $\chi_{\varepsilon}(t):=\frac{1}{\eps}\chi\left(\frac{t}{\eps} \right)$. Set $\phi_{\eps}(t,x)=\chi_{\eps}(t)\psi_{\eps}(x)$. The mollified functions $\phi_{\eps}\ast g_j$ lie in $C^{\infty}((0,T-\eps)\times \T^d)$. To remedy the shortened domain in the time variable we introduce the scaling $t\mapsto \frac{T-\varepsilon}{T}t$. Set
		\begin{align*}
			h_{\eps,j}(t,x):=\left(\frac{T}{T-\eps}\right)^2(\phi_{\eps}\ast g_j)\left(\frac{T-\eps}{T}t,x \right)\text{ for all }(t,x)\in(0,T)\times \T^d.
		\end{align*}
		These functions are smooth on $(0,T)\times \T^d$ and satisfy $\mathcal{B}_Eh_{\eps,j}(t,x)=\left(\phi_{\eps}\ast\mathcal{B}_Eg_j \right)\left(\frac{T-\eps}{T}t,x \right)$. For all $\varepsilon,j$ It is then straightforward to check that $\mathcal{B}_Eh_{\eps,j}$ is uniformly continuous on $[0,T]\times \T^d$. Moreover, for fixed $j$ we have
		\begin{align*}
			&\left\|\mathcal{B}_Eh_{\eps,j}-\mathcal{B}_Eg_j\right\|_{L^1((0,T)\times\T^d)}\\
			\leq &\frac{T}{T-\eps}\left\|\phi_{\eps}\ast\mathcal{B}_Eg_j-\mathcal{B}_Eg_j \right\|_{L^1((0,T-\eps)\times\T^d)}+\left\|\mathcal{B}_Eg_j\left(\frac{T-\eps}{T}\cdot,\cdot \right)-\mathcal{B}_Eg_j \right\|_{L^1((0,T)\times\T^d)}\\
			\overset{\eps\rightarrow 0}{\rightarrow}&0.
		\end{align*}
		Fix $(t,x)\in(0,T)\times \T^d$ and estimate
		\begin{align*}
			\dist\left(\mathcal{B}_Eh_{\eps,j}(t,x),K_{(t,x)}-\sigma(t,x) \right)
			\leq &\left\|(\tau,y)\mapsto\dist\left(\mathcal{B}_Eg_j(\tau,y),K_{(\tau,y)}-\sigma(\tau,y) \right)\right\|_{L^{\infty}((0,T)\times \T^d)}\\
			&+\int\limits_{\R^{d+1}}^{}\phi_{\eps}\left(\frac{T-\eps}{T}(t-\tau),x-y\right)d_H\left(K_{(t,x)}-\sigma(t,x),K_{(\tau,y)}-\sigma(\tau,y) \right)\dd\tau\dy\\
			\overset{\eps\rightarrow 0}{\rightarrow}&\left\|(\tau,y)\mapsto\dist\left(\mathcal{B}_Eg_j(\tau,y),K_{(\tau,y)}-\sigma(\tau,y) \right)\right\|_{L^{\infty}((0,T)\times \T^d)}.
		\end{align*}
		Here, we used Jensen's inequality, the estimate (\ref{eq:hdcontinuity}), and the fact that $|t-\tau|\leq 2\eps$ for all $\tau$ in the support of $\chi_{\eps}\left(\frac{T-\eps}{T}(t-\cdot) \right)$. Note that this convergence is uniform in $(t,x)$.\\
		Therefore, we can choose a subsequence (not-relabeled) in $j\in \N$ and corresponding $\eps_j$ such that
		\begin{align}
			\left\|\mathcal{B}_Eh_{\eps_j,j}-\mathcal{B}_Eg_j\right\|_{L^1((0,T)\times\T^d)}&\leq \frac{1}{j},\nonumber\\
			\left\|(t,x)\mapsto \dist\left(\mathcal{B}_Eh_{\eps_j,j}(t,x),K_{(t,x)}-\sigma(t,x) \right)\right\|_{L^{\infty}((0,T)\times \T^d)}&\leq \frac{1}{j},\nonumber\\
			\eps_j< \inf\{t\in[0,T]\,:\, \exists\,x\in \T^d\text{ s.t. }&(t,x)\in U_j\}.\label{eq:minimaltimemollification}
		\end{align}
		Define the sequence
		\begin{align*}
			(\rho_j,m_j,M_j,\tilde{Q}_j):=\sigma+\mathcal{B}_Eh_{\eps_j,j}\in C\big([0,T]\times \T^d\big)\cap C^{\infty}\big((0,T)\times \T^d\big).
		\end{align*}
		This sequence satisfies
		\begin{align*}
			\tilde{Q}_j&\geq \frac{2}{d}e_{\operatorname{kin}}(\rho_j,m_j,M_j)+\rho_j^{\gamma}-\frac{1}{j}L-\frac{1}{j},\\
			\left|\tilde{Q}_j-\frac{2}{d}\langle\nu,e\rangle\right|&\leq \frac{1}{j},\\
			\rho_j&\geq \frac{\eta}{2}-\frac{1}{j}
		\end{align*}
		on $(0,T)\times \T^d$, where $L$ is the minimal Lipschitz constant of the convex function $(\rho,m,M)\mapsto\frac{2}{d}e_{\operatorname{kin}}(\rho,m,M)+\rho^{\gamma}$ on $[-(\bar{R}+1),\bar{R}+1]^{N-1}$. This is valid since the values of $\sigma+\mathcal{B}_Eh_{\eps_j,j}$ are contained in $[-(\bar{R}+1),\bar{R}+1]^N$. Moreover, Lemma \ref{lem:initialvalueproperties} yields that
		\begin{align*}
			\|\rho_j(0,\cdot)-\rho_0\|_{L^{\gamma}(\T^d)}
			\leq &\left\|(\phi_{\eps_j}\ast \langle \tilde{\nu},\rho\rangle)(0,\cdot)-\int\limits_{0}^{\eps_j}(\langle\tilde{\nu},\rho\rangle(0,\cdot)\ast\psi_{\eps_j})\chi_{\eps_j}(-s)\ds\right\|_{L^{\gamma}(\T^d)}\\
			&+\|\langle\tilde{\nu},\rho\rangle(0,\cdot)\ast \psi_{\eps_j}-\langle\tilde{\nu},\rho\rangle(0,\cdot)\|_{L^{\gamma}(\T^d)}
			\\
			\leq &\underset{0\leq s\leq \eps_j}{\sup}\|\langle\tilde{\nu},\rho\rangle(s,\cdot)-\langle\tilde{\nu},\rho\rangle(0,\cdot)\|_{L^{\gamma}(\T^d)}\\
			&+\|\langle\tilde{\nu},\rho\rangle(0,\cdot)\ast \psi_{\eps_j}-\langle\tilde{\nu},\rho\rangle(0,\cdot)\|_{L^{\gamma}(\T^d)}
			\\
			\rightarrow &0,
		\end{align*}
		where we used the smallness condition on $\varepsilon_j$ from (\ref{eq:minimaltimemollification}). An analogous calculation yields
		\begin{align*}
			\|m_j(0,\cdot)-m_0\|_{L^2(\T^d)}\rightarrow 0.
		\end{align*}
		With a slightly adjusted estimate one can show that $\rho_j(0,\cdot)\rightharpoonup \rho_0$ in $L^{\gamma}(\T^d)$ and $m_j(0,\cdot)\rightharpoonup m_0$ in $L^2(\T^d)$, but for that the strong convergence in $L^{\gamma}(\T^d)$ and $L^2(\T^d)$ is not needed. We give the details only for $\rho_j$. Let $\varphi\in C^{\infty}(\T^d)$. Then
		\begin{align*}
			&\left|\int\limits_{\T^d}^{}\left(\rho_j(0,x)-\rho_0(x) \right)\varphi(x)\dx \right|\\
			\leq &\left|\int\limits_{\T^d}^{}\left((\phi_{\eps_j}\ast\langle\tilde{\nu},\rho\rangle )(0,x)-\int\limits_{0}^{\eps_j}(\psi_{\eps_j}\ast\langle\tilde{\nu},\rho\rangle)(0,x)\chi_{\eps_j}(-s)\ds \right)\varphi(x)\dx \right|\\
			&+\left|\int\limits_{\T^d}^{}\left((\psi_{\eps_j}\ast\langle\tilde{\nu},\rho\rangle)(0,x)-\langle\tilde{\nu},\rho\rangle(x) \right)\varphi(x)\dx \right|
			\\
			\leq &\underset{0\leq s\leq \eps_j}{\sup}\left|\int\limits_{\T^d}^{}\left(\langle\tilde{\nu},\rho\rangle(s,x)-\langle\tilde{\nu},\rho\rangle(0,x) \right)\varphi(x)\dx \right|+\eps_j\|D\varphi\|_{L^{\infty}(\T^d)}2\|\langle\tilde{\nu},\rho\rangle\|_{L^{\infty}((0,T)\times \T^d)}\\
			&+\|\langle\tilde{\nu},\rho\rangle(0,\cdot)\ast\psi_{\eps_j}-\langle\tilde{\nu},\rho\rangle(0,\cdot)\|_{L^{\gamma}(\T^d)}\cdot\|\varphi\|_{L^{\gamma'}(\T^d)}
			\\
			\rightarrow & 0,
		\end{align*}
		where we also used Taylor's theorem. Since $(\rho_j(0,\cdot)-\rho_0)$ is uniformly bounded in $L^{\gamma}(\T^d)$, this yields the assertion. The weak convergence of the initial values is only important if we want to consider not necessarily admissible $\nu$ when proving Theorem \ref{theo:weakmainresult}.\\
		Note that $\sigma$ is $\mathcal{A}$-free, since $\langle\tilde{\nu},\operatorname{id}\rangle$ is $\mathcal{A}$-free. Thus, the sequence of continuous functions
		\begin{align*}
			z_j:=(\rho_j,m_j,M_j,Q_j):=\left(\rho_j,m_j,M_j,\tilde{Q}_j+\frac{1}{j}(L+2)\right)\text{ with }j\geq \frac{4}{\eta}
		\end{align*}
		has the properties
		\begin{align*}
			\mathcal{A}_Ez_j=0,\ \ \ \ z_j\overset{Y}{\rightharpoonup}\tilde{\nu},\ \ \ \ \rho_j\geq \frac{\eta}{4},\ \ \ \ Q_j>\frac{2}{d}e_{\operatorname{kin}}(\rho_j,m_j,M_j)+\rho_j^{\gamma},\ \ \ \ \left|Q_j-\frac{2}{d}\langle\nu,e\rangle\right|\leq\frac{1}{j}(L+3).
		\end{align*}
		It follows from Lemma \ref{lem:lemmatwopointtwo} that $e_{\operatorname{kin}}(\rho_j,m_j,M_j)$ bounds $\frac{|m_j|^2}{2\rho_j}$ and $\frac{d}{2(d-1)}|M_j|$. Hence, the sequence $(z_j)$ is uniformly bounded. The desired convergences of $\rho_j(0,\cdot)$ and $m_j(0,\cdot)$ in $L^{\gamma}$ and $L^2$, respectively, have been proven above. So, taking another subsequence in order to lessen the bounds to be smaller than $\frac{1}{j}$ finishes the proof.
	\end{proof}
	We have used the following auxiliary result about the Hausdorff distance of certain sets related to the set $K$ from the above proof.
	\begin{lem}\label{lem:hddistance}
		Let $\eta,\tilde{R},\bar{R},\sigma,\gamma,N$ be as in the proof of Proposition \ref{theo:stepone}. Consider the convex and compact set
		\begin{align}
			\mathcal{C}:=&\left\{(\rho,m,M,Q) \,:\, \frac{2}{d}e_{\operatorname{kin}}(\rho,m,M)+\rho^{\gamma}\leq Q\right\}\cap \left(\left[\frac{\eta}{2},\infty\right)\times \R^d\times S_0^d\times\R^+\right)\\
			&\cap \left[-\bar{R},\bar{R} \right]^N.
		\end{align}
		Let $f\colon A\to [\eta^{\gamma},\tilde{R}+\|\sigma\|_{L^{\infty}}]$ be a continuous function on a compact set $A\subset \R^{d+1}$.\\
		Then if the sets
		\begin{align}
			K_x:=\mathcal{C}\cap\{(\rho,m,M,Q)\,:\,Q= f(x) \}
		\end{align}
		are nonempty for all $x\in A$, they have the property that for all $\varepsilon>0$ there exists some $\delta>0$ such that
		\begin{align}
			d_H(K_x,K_y)\leq \varepsilon\text{ if }|x-y|\leq \delta.
		\end{align}
		Here, $d_H$ denotes the Hausdorff distance.
	\end{lem}
	\begin{proof}
		First, note that the vectors
		\begin{align}
			z_1:=&\left(\frac{\eta}{2},0,0,\left(\frac{\eta}{2} \right)^\gamma \right),\\
			z_2:=&\left(\frac{\eta}{2},0,0,\bar{R} \right)
		\end{align}
		lie in the convex and compact set $\mathcal{C}$. 
		For all $x\in A$ it holds that $\tilde{R}+\|\sigma\|_{L^{\infty}}\geq f(x)\geq \eta^{\gamma}$, hence for every $(\rho,m,M,Q)\in K_x$ it holds that $\tilde{R}+\|\sigma\|_{L^{\infty}}\geq Q\geq \eta^{\gamma}$. Thus, for any $x\in A$ and any $z\in K_x$, the line through $z_1=(\rho_1,m_1,M_1,Q_1)$ and $z=(\rho,m,M,Q)$ has $Q$-slope at least
		\begin{align}
			\frac{Q-Q_1}{|(\rho_1,m_1,M_1)-(\rho,m,M)|}\geq \frac{\eta^{\gamma}-\left(\frac{\eta}{2} \right)^{\gamma}}{2\sqrt{N}\bar{R} }=:a>0.
		\end{align}
		Similarly, for any $z\in \mathcal{C}_x$, the line through $z=(\rho,m,M,Q)$ and $z_2=(\rho_2,m_2,M_2,Q_2)$ has $Q$-slope at least
		\begin{align}
			\frac{Q_2-Q}{|(\rho,m,M)-(\rho_2,m_2,M_2)|}\geq \frac{\tilde{R}+\|\sigma\|_{L^{\infty}}+\eta^{\gamma}-\left(\frac{\eta}{2} \right)^{\gamma}-\left(\tilde{R}+\|\sigma\|_{L^{\infty}}\right)}{2\sqrt{N}\bar{R}}=a,
		\end{align}
		where we used the definition $\bar{R}=\tilde{R}+\|\sigma\|_{L^{\infty}}+\eta^{\gamma}-\left(\frac{\eta}{2} \right)^{\gamma}$. Now fix $x,y\in A$ and assume without loss of generality that $f(x)<f(y)$. Let $z\in K_y$ and denote by $z_x\in \mathcal{C}_x$ the intersection of the line through $z_1$ and $z$ with the set $K_x$. Such a point $z_x$ exists due to the convexity of $K$. Then
		\begin{align}
			\dist(z,K_x)\leq |z-z_x|\leq \sqrt{|f(x)-f(y)|^2+\frac{1}{a^2}|f(x)-f(y)|^2}=|f(x)-f(y)|\sqrt{1+\frac{1}{a^2}}.
		\end{align}
		Similarly, for $\bar{z}\in K_x$ denote by $\bar{z}_y\in K_x$ the intersection of the line through $z_2$ and $\bar{z}$ with the set $K_y$. Then
		\begin{align}
			\dist(\bar{z},K_y)\leq |\bar{z}-\bar{z}_y|\leq |f(x)-f(y)|\sqrt{1+\frac{1}{a^2}}.
		\end{align}
		Note that $f$ is uniformly continuous as $A$ is compact. So, for every $\eps>0$ there exists $\delta>0$ such that $|f(x)-f(y)|\leq \frac{\eps}{\sqrt{1+\frac{1}{a^2}}}$ for all $x,y\in A$ such that $|x-y|\leq \delta$. Thus, in summary, 
		\begin{align}
			d_H(K_x,K_y)=\max\left\{\underset{z\in K_x}{\max}\dist(z,K_y),\underset{\bar{z}\in K_y}{\max}\dist(\bar{z},K_x) \right\}\leq |f(x)-f(y)|\sqrt{1+\frac{1}{a^2}}\leq \eps
		\end{align}
		for all $x,y\in A$ with $|x-y|\leq \delta$. 
	\end{proof}
	Now we improve the sequence from Proposition \ref{theo:stepone} such that the space integral over the generalized pressure component is maximal at time $t=0$. This will eventually lead to the energy admissibility of the corresponding weak solutions in Proposition \ref{prop:stepthree}, below.
	\begin{prop}\label{prop:steptwo}
		Assume the lift $\tilde{\nu}$ of an admissible measure-valued solution $\nu$ on $(0,T)\times \T^d$ with initial data $(\rho_0,m_0)\in L^{\infty}(\T^d)$ can be generated by a sequence of subsolutions $(\rho_j,m_j,M_j,Q_j)\in C([0,T]\times \T^d)$ with the properties:
		\begin{align*}
			\|(\rho_j,m_j,M_j,Q_j)\|_{L^{\infty}}&\leq R,\\
			\rho_j&\geq \eta\text{ on }(0,T)\times \T^d,\\
			\frac{2}{d}e_{\operatorname{kin}}(\rho_j,m_j,M_j)+\rho_j^{\gamma}&< Q_j\text{ on }(0,T)\times 
			\T^d,\\
			\left|Q_j(t,x)-\frac{2}{d}\langle\nu_{(t,x)},e\rangle\right|&\leq \frac{1}{j}\text{ for all }(t,x)\in[0,T]\times \T^d,\\
			(t,x)&\mapsto \langle\nu_{(t,x)},e\rangle\text{ is continuous on }[0,T]\times \T^d,\\
			\rho_j(0,\cdot)&\rightarrow \rho_0\text{ in }L^{\gamma}(\T^d),\\
			m_j(0,\cdot)&\rightarrow m_0\text{ in }L^{2}(\T^d)
		\end{align*}
		for some $R,\eta>0$.\\
		Then $\tilde{\nu}$ is generated by a sequence of uniformly bounded subsolutions $(\tilde{\rho}_j,\tilde{m}_j,\tilde{M}_j,\tilde{Q}_j)\in C((0,T)\times \T^d)$ with the property that:
		\begin{align*}
			\tilde{\rho}_j,\tilde{Q}_j&\in C([0,T]\times \T^d),\\
			\tilde{m}_j&\in C([0,T],L^2_{\operatorname{w}}(\T^d)),\\
			\tilde{\rho}_j&\geq \eta\text{ on }(0,T)\times \T^d,\\
			\frac{2}{d}e_{\operatorname{kin}}(\tilde{\rho}_j,\tilde{m}_j,\tilde{M}_j)+\tilde{\rho}_j^{\gamma}&< \tilde{Q}_j\text{ on } (0,T)\times 
			\T^d,\\
			e(\tilde{\rho}_j(0,x),\tilde{m}_j(0,x))&=\frac{d}{2}\tilde{Q}_j(0,x)\text{ for a.e.~}x\in\T^d,\\
			\int\limits_{\T^d}^{}\tilde{Q}_j(t,x)\dx&\leq \int\limits_{\T^d}^{}\tilde{Q}_j(0,x)\dx\text{ for all }t\in[0,T],\\
			\tilde{\rho}_j(0,\cdot)&\rightarrow \rho_0\text{ in }L^{\gamma}(\T^d),\\
			\tilde{m}_j(0,\cdot)&\rightarrow m_0\text{ in }L^{2}(\T^d).
		\end{align*}
	\end{prop}
	\begin{proof}
		Define the sequence
		\begin{align}
			\tilde{Q}_j(t,x):=Q_j(t,x)+\frac{1}{j}+\left(\int\limits_{\T^d}^{}\frac{2}{d}\langle\nu,e\rangle-Q_j\dx\right)(t)\text{ for all }(t,x)\in[0,T]\times \T^d.\label{eq:qcomponenttildesequence}
		\end{align}
		Then $(\rho_j,m_j,M_j,\tilde{Q}_j)\in C([0,T]\times \T^d)$ is again a subsolution and generates $\tilde{\nu}$ when $j\rightarrow \infty$ since $\|Q_j-\tilde{Q}_j\|_{L^{\infty}((0,T)\times \T^d)}\leq \frac{2}{j}$. Moreover, on $(0,T)\times \T^d$ it holds that
		\begin{align*}
			\tilde{Q}_j\geq Q_j>\frac{2}{d}e_{\operatorname{kin}}(\rho_j,m_j,M_j)+\rho_j^{\gamma}.
		\end{align*}
		Thus, we can apply Proposition \ref{prop:proposition22} below to the sequence $(\rho_j,m_j,M_j,\tilde{Q}_j)$ and obtain a doubly indexed sequence of subsolutions $(\rho_j,m_j^n,M_j^n,\tilde{Q}_j)$ with $(m_j^n,M_j^n)\in C((0,T)\times \T^d)$ and $m_j^n\in C([0,T],L^2_{\operatorname{w}}(\T^d))$ satisfying
		\begin{align*}
			\frac{2}{d}e_{\operatorname{kin}}(\rho_j,m_j^n,M_j^n)+\rho_j^{\gamma}&<\tilde{Q}_j\text{ on }(0,T)\times \T^d,\\
			\frac{\left|m_j^n(0,x)\right|^2}{d\rho_j(0,x)}+\rho_j(0,x)^{\gamma}&=\tilde{Q}_j(0,x)\text{ for a.e.~}x\in\T^d.
		\end{align*}
		As $m_j^n(t)\overset{n\rightarrow\infty}{\rightharpoonup}m_j(t)$ in $L^2(\T^d)$ for all $t\in [0,T]$, we get for every fixed $j$ using Fatou's lemma that
		\begin{align*}
			\int\limits_{(0,T)}^{}\int\limits_{\T^d}^{}|m_j|^2\dx\dt\leq \underset{n\rightarrow\infty}{\liminf}\int\limits_{(0,T)}^{}\int\limits_{\T^d}^{}|m_j^n|^2\dx\dt.
		\end{align*}
		Hence, we may choose a subsequence $n(j)$ such that
		\begin{align*}
			\int\limits_{(0,T)}^{}\int\limits_{\T^d}^{}|m_j|^2\dx\dt\leq \int\limits_{(0,T)}^{}\int\limits_{\T^d}^{}\left|m_j^{n(j)}\right|^2\dx\dt+\frac{1}{j}.
		\end{align*}
		Since $\tilde{Q}_j-\rho_j^{\gamma}\geq \frac{2}{d}e_{\operatorname{kin}}\left(\rho_j,m_j^{n(j)},M_j^{n(j)}\right)\geq \frac{\left|m_j^{n(j)}\right|^2}{d\rho_j}$ on $(0,T)\times\T^d$, we therefore obtain
		\begin{align*}
			&\left|\int\limits_{(0,T)}^{}\int\limits_{\T^d}^{} \left|m_j^{n(j)}\right|^2-|m_j|^2\dx\dt\right|\\
			\leq& \int\limits_{(0,T)}^{}\int\limits_{\T^d}^{}\left|m_j^{n(j)}\right|^2+\frac{1}{j}-|m_j|^2\dx\dt+\frac{1}{j}\\
			\leq &\, d\int\limits_{(0,T)}^{}\int\limits_{\T^d}^{}(\tilde{Q}_j-\rho_j^{\gamma})\rho_j-\frac{2}{d}e_{\operatorname{kin}}(\rho_j,m_j,M_j)\rho_j\,\dx\dt\\
			& +d\int\limits_{(0,T)}^{}\int\limits_{\T^d}^{}\frac{2}{d}e_{\operatorname{kin}}(\rho_j,m_j,M_j)\rho_j-\rho_j\frac{|m_j|^2}{d\rho_j}\dx\dt+\frac{2}{j}.
		\end{align*}
		For treating the first summand observe that $\rho_j(\tilde{Q}_j-\rho_j^{\gamma}-\frac{2}{d}e_{\operatorname{kin}}(\rho_j,m_j,M_j))$ is uniformly bounded and $\e_{\operatorname{kin}}$ is continuous. So, the Fundamental Theorem for Young Measures, cf.~for example \cite{FM99} Theorem 2.2, implies that
		\begin{align*}
			\int\limits_{(0,T)}^{}\int\limits_{\T^d}^{}\rho_j\left(\tilde{Q}_j-\rho_j^{\gamma}-\frac{2}{d}e_{\operatorname{kin}}(\rho_j,m_j,M_j)\right) \dx\dt\rightarrow \int\limits_{(0,T)}^{}\int\limits_{\T^d}^{}\int\limits_{\R^m}^{}\rho\left(Q-\left(\frac{2}{d}e_{\operatorname{kin}}(\rho,m,M)+\rho^{\gamma}\right)\dd\tilde{\nu}_{(t,x)}\right)\dx\dt=0.
		\end{align*}
		The last equality follows from Lemma \ref{lem:lemmatwopointtwo} and the support properties of $\tilde{\nu}$. For the second summand, we estimate as in the proof of Theorem 16 in \cite{SW12}:
		\begin{align*}
			\left|\int\limits_{(0,T)}^{}\int\limits_{\T^d}^{}\rho_j\left(\frac{2}{d}e_{\operatorname{kin}}(\rho_j,m_j,M_j)-\frac{|m_j|^2}{d\rho_j}\right)\dx\dt\right|
			&=\int\limits_{(0,T)}^{}\int\limits_{\T^d}^{}\rho_j\lambda_{\operatorname{max}}\left(\frac{m_j\ocircle m_j}{\rho_j}-M_j\right)\dx\dt\\
			&\leq R\cdot C\int\limits_{(0,T)}^{}\int\limits_{\T^d}^{}\left| \frac{m_j\ocircle m_j}{\rho_j}-M_j\right|\dx\dt\rightarrow 0.
		\end{align*}
		The convergence to zero on the right hand side is derived by using the Fundamental Theorem for Young Measures.\\
		Now choose another subsequence $n(j)$ (not relabeled) if necessary satisfying
		\begin{align*}
			\underset{t\in(0,T)}{\sup}\int\limits_{\T^d}^{}m_j\cdot\left(m_j^{n(j)}-m_j\right)\dx\leq \frac{1}{j}.
		\end{align*}
		So, we get
		\begin{align*}
			\int\limits_{(0,T)}^{}\int\limits_{\T^d}^{}\left|m_j^{n(j)}-m_j\right|^2\dx\dt&=\int\limits_{(0,T)}^{}\int\limits_{\T^d}^{}\left|m_j^{n(j)}\right|^2-|m_j|^2\dx\dt-2\int\limits_{(0,T)}^{}\int\limits_{\T^d}^{}m_j\cdot\left(m_j^{n(j)}-m_j\right)\dx\dt\\
			&\leq \left\|m_j^{n(j)}\right\|_{L^2}- \|m_j\|_{L^2}+\frac{2}{j},
		\end{align*}
		which tends to zero for $j\rightarrow\infty$ as we have already seen. Considering now the $M$-components, we estimate
		\begin{align*}
			\int\limits_{(0,T)}^{}\int\limits_{\T^d}^{}\left|M_j^{n(j)}-M_j\right|\dx\dt&\leq \int\limits_{(0,T)}^{}\int\limits_{\T^d}^{}\left|M_j^{n(j)}-\frac{m_j^{n(j)}\ocircle m_j^{n(j)}}{\rho_j}\right|\dx\dt+\int\limits_{(0,T)}^{}\int\limits_{\T^d}^{}\left|M_j-\frac{m_j^{n(j)}\ocircle m_j^{n(j)}}{\rho_j}\right|\dx\dt
		\end{align*}
		For the first summand observe that
		\begin{align*}
			\int\limits_{(0,T)}^{}\int\limits_{\T^d}^{}\left|M_j^{n(j)}-\frac{m_j^{n(j)}\ocircle m_j^{n(j)}}{\rho_j}\right|\dx\dt&\leq d\cdot C\int\limits_{(0,T)}^{}\int\limits_{\T^d}^{}\lambda_{\operatorname{max}}\left(\frac{m_j^{n(j)}\ocircle m_j^{n(j)}}{\rho_j}-M_j^{n(j)} \right)\dx\dt\\
			&\leq d\cdot C\left|\int\limits_{(0,T)}^{}\int\limits_{\T^d}^{}\tilde{Q}_j-\rho_j^{\gamma}-\frac{|m_j^{n(j)}|^2}{d\rho_j}\dx\dt\right|\\
			&\leq \frac{C}{\eta}\int\limits_{(0,T)}^{}\int\limits_{\T^d}^{}d\rho_j(\tilde{Q}_j-\rho_j^{\gamma})-|m_j|^2\dx\dt+\frac{C}{\eta}\left|\left\|m_j^{n(j)}\right\|^2_{L^2}-\|m_j\|^2_{L^2}\right|\\
			&\rightarrow 0.
		\end{align*}
		For the second summand we estimate
		\begin{align*}
			&\int\limits_{(0,T)}^{}\int\limits_{\T^d}^{}\left|M_j-\frac{m_j^{n(j)}\ocircle m_j^{n(j)}}{\rho_j}\right|\dx\dt\\
			\leq & \int\limits_{(0,T)}^{}\int\limits_{\T^d}^{}\left|M_j-\frac{m_j\ocircle m_j}{\rho_j}\right|\dx\dt+\frac{C}{\eta}\left\|m_j\right\|_{L^2}\cdot\left\|m_j^{n(j)}-m_j\right\|_{L^2}+\frac{C}{\eta}\left\|m_j^{n(j)}\right\|_{L^2}\cdot\left\|m_j^{n(j)}-m_j\right\|_{L^2}\\
			\rightarrow & 0.
		\end{align*}
		In summary we have shown that $m_j^{n(j)}-m_j\rightarrow 0$ in $L^2((0,T)\times \T^d)$ and $M_j^{n(j)}-M_j\rightarrow 0$ in $L^1((0,T)\times \T^d)$. So, the sequence $\left(\rho_j,m_j^{n(j)},M_j^{n(j)},\tilde{Q}_j\right)$ generates the Young measure $\tilde{\nu}$.\\
		Note that
		\begin{align*}
			\int\limits_{\T^d}^{}\tilde{Q}_j(t,x)\dx\leq \frac{1}{j}+\frac{2}{d}\int\limits_{\T^d}^{}e(\rho_0(x),m_0(x))\dx=\frac{1}{j}+\frac{2}{d}\int\limits_{\T^d}^{}\langle\nu_{(0,x)},e\rangle\dx=\int\limits_{\T^d}^{}\tilde{Q}_j(0,x)\dx,
		\end{align*}
		where we used Lemma \ref{lem:initialvalueproperties}.\\
		Proposition \ref{prop:proposition22} ensures that $m_j^{n(j)}(t=0)-m_j(t=0)\rightharpoonup 0$ in $L^2(\T^d)$. On the other hand, we can estimate for a.e.~$x\in\T^d$
		\begin{align*}
			\frac{\left|m_j^{n(j)}(0,x)\right|^2}{d\rho_j(0,x)}+\rho_j^{\gamma}(0,x)&=\tilde{Q}_j(0,x)\geq Q_j(0,x)\\
			&\geq \frac{2}{d}e_{\operatorname{kin}}(\rho_j(0,x),m_j(0,x),M_j(0,x))+\rho_j^{\gamma}(0,x)\\
			&\geq \frac{|m_j(0,x)|^2}{d\rho_j(0,x)}+\rho_j^{\gamma}(0,x)
		\end{align*}
		by using the continuity of $\rho_j,m_j,M_j,Q_j,\tilde{Q}_j$. Hence, it holds that $\left|m_j^{n(j)}(0,\cdot) \right|\geq |m_j(0,\cdot)|$ a.e.~on $\T^d$. Therefore, we obtain
		\begin{align*}
			&\left|\int\limits_{\T^d}^{}\left|m_j^{n(j)}(0) \right|^2\dx-\int\limits_{\T^d}^{}\left|m_j(0) \right|^2\dx \right|\\
			\leq &dR\left|\int\limits_{\T^d}^{}e(\rho_0,m_0)-\rho_j(0)^{\gamma}\dx-\int\limits_{\T^d}^{}\frac{|m_j(0)|^2}{d\rho_j(0)}\dx \right|+\frac{dR}{j}\\
			\leq &dR\left|\|\rho_j(0,\cdot)\|_{L^{\gamma}(\T^d)}^{\gamma}-\|\rho_0\|_{L^{\gamma}(\T^d)}^{\gamma} \right|+\frac{2R^2}{\eta}\|m_j(0,\cdot)-m_0\|_{L^2(\T^d)}\\
			&+\frac{R}{\eta^2}\|m_0\|_{L^{\frac{2\gamma}{\gamma-1}}(\T^d)}^2\|\rho_j(0,\cdot)-\rho_0\|_{L^{\gamma}(\T^d)}+\frac{dR}{j},
		\end{align*}
		which tends to zero as $j\rightarrow\infty$, since $\|m_0\|_{L^{\frac{2\gamma}{\gamma-1}}(\T^d)}\leq R$. So, we obtain that $m_j^{n(j)}(t=0)\rightarrow m_0$ in $L^2(\T^d)$.\\
		Now define
		\begin{align*}
			(\tilde{\rho}_j,\tilde{m}_j,\tilde{M}_j,\tilde{Q}_j):=\left(\rho_j,m_j^{n(j)},M_j^{n(j)},\tilde{Q}_j\right).
		\end{align*}
		Note that $\tilde{Q}_j\leq Q_j+\frac{2}{j}\leq R+2$ and $e_{\operatorname{kin}}\left(\rho_j,m_j^{n(j)},M_j^{n(j)} \right)\leq \frac{d}{2}\tilde{Q}_j$. This implies the uniform boundedness of the sequence $(\tilde{\rho}_j,\tilde{m}_j,\tilde{M}_j,\tilde{Q}_j)$ by Lemma \ref{lem:lemmatwopointtwo}, which finishes the proof.
	\end{proof}
	\begin{rem}
		Note that the arguably quite restrictive condition in Theorem \ref{theo:mainresult} on the support of $\tilde{\nu}$ being concentrated in the $Q$-component becomes important exactly when we prove the above Proposition \ref{prop:steptwo}. This condition ensures the desired properties of the $Q$-component of the updated sequence defined in (\ref{eq:qcomponenttildesequence}).
	\end{rem}
	In the above proof we heavily made use of the following proposition.
	\begin{prop}\label{prop:proposition22}
		Let $(\rho,m,M,Q)\in C([0,T]\times \T^d)$ be a subsolution such that
		\begin{align*}
			\frac{2}{d}e_{\operatorname{kin}}(\rho,m,M)+\rho^{\gamma}&<Q\text{ on }(0,T)\times \T^d,\\
			\rho\geq \eta\text{ on }[0,T]\times \T^d,
		\end{align*}
		for some $\eta>0$ fixed.\\
		Then there exists a sequence $(m^n,M^n)\in C((0,T)\times \T^d)$ such that $(\rho,m^n,M^n,Q)$ is a subsolution and
		\begin{align*}
			m^n&\rightarrow m\text{ in }C([0,T],L^2_{\operatorname{w}}(\T^d)),\\
			\frac{2}{d}e_{\operatorname{kin}}(\rho,m^n,M^n)+\rho^{\gamma}&<Q\text{ on }(0,T)\times \T^d,\\
			\frac{|m^n(0,x)|^2}{d\rho(0,x)}+\rho(0,x)^{\gamma}&=Q(0,x)\text{ for a.e.~}x\in\T^d
		\end{align*}
		for all $n\in \N$.
	\end{prop}
	The proof largely follows the one of Proposition 5.1 in~\cite{DLSz10}. Before we give the proof, we need some preparation. In the sequel let $(\rho,m,M,Q)$ be fixed as in the statement of Proposition~\ref{prop:proposition22}.\\
	Define the space $X_0$ as the set of vector fields $\bar m\in C([0,T)\times\mathbb{\T^d},\R^d)$ such that there exists a matrix field $\bar M\in C([0,T)\times\T^d,\mathcal{S}_0^d)$ with the following properties:
	\begin{itemize}
		\item $\partial_t\rho+\diverg \bar m=0$, $\partial_t\bar m+\diverg\bar M+\nabla Q=0$ in the sense of distributions;
		\item $\supp (\bar m-m,\bar M-M)\subset[0,T)\times\T^d$;
		\item $\frac2d e_{kin}(\rho, \bar m, \bar M)+\rho^\gamma <Q$ on $[0,T)\times \T^d$.
	\end{itemize}
	The bound on $e_{kin}$, together with the fact that $\rho$ is bounded below by a positive constant, implies that $X_0$ is a bounded subset of $C([0,T],L^2_w(\T^d))$. Hence, we observe that the closure $X$ of $X_0$ with respect to the topology of $C([0,T],L^2_w(\T^d))$ is metrizable by a metric $\mathcal{D}$, and thus $(X,\mathcal{D})$ is a complete metric space.\\		
	On $X$, let us define the functional $I:X\to\R$ by
	\begin{equation*}
		I(\bar m):=\int_{\T^d}\left(\frac{|\bar m(0,x)|^2}{d\rho(0,x)}+\rho(0,x)^\gamma-Q(0,x)\right)dx.
	\end{equation*}
	Clearly, $I\leq 0$ with equality if and only if $\frac{|\bar m(0,x)|^2}{d\rho(0,x)}+\rho(0,x)^\gamma=Q(0,x)$.\\		
	The proof of the perturbation property (Proposition 3.1) in~\cite{EJT20} then yields the following statement:
	\begin{prop}\label{timezeroperturbation}
		With the previous notation, let $T>\epsilon>0$ and $\alpha>0$. There exists $\beta>0$ such that the following is true: If $\bar m\in X_0$ such that $I(\bar m)<-\alpha$, then there exists a sequence $(\bar m_k)_{k\in\N}\subset X_0$ and corresponding matrix fields $(\bar M_k)_k\subset C([0,T)\times\T^d,\mathcal{S}_0^d)$ such that
		\begin{itemize}
			\item $\supp(\bar m-\bar m_k,\bar M-\bar M_k)\subset[0,\epsilon)\times\T^d$;
			\item $\bar m_n\to\bar m$ with respect to $\mathcal{D}$;
			\item $I(\bar m_k)\geq I(\bar m)+\beta$ for all $k\in \N$.
		\end{itemize}
	\end{prop}
	The proof proceeds as the one of Proposition 3.1 in~\cite{EJT20}, where the perturbations do not affect $\rho$ and $Q$. The only difference to said proof is that we include perturbations only up to time $\epsilon$ to ensure the support condition.
	\begin{proof}[Proof of Proposition~\ref{prop:proposition22}]
		The proof is now essentially a line-by-line copy of the proof of Proposition 5.1 in~\cite{DLSz10}, but we give it here for the reader's convenience.\\			
		Set $(\bar m_0,\bar M_0)=(m,M)$ and define recursively a sequence in $X_0$ in the following way: Let $\eta_\delta$ be a standard mollifier in $\T^d$. If $(\bar m_n,\bar M_n)$ is given, we choose a sequence $(\gamma_n)_{n\in\N}\subset\R^+$ such that $\gamma_n<2^{-n}$ for every $n\in\N$ and such that
		\begin{equation}\label{gamma}
			\|\bar m_{n}(t=0)-(\bar m_n*\eta_{\gamma_n})(t=0)\|_{L^2}<2^{-n}.
		\end{equation}
		Setting $\alpha_n:=-\frac12 I(\bar m_n)$ and denoting $\beta_n$ the corresponding perturbation improvement, we apply Proposition~\ref{timezeroperturbation} to $\bar m_n$ with $\epsilon:=2^{-n}T$; this yields $(\bar m_{n+1},\bar M_{n+1})$ such that
		\begin{enumerate}
			\item[(i)] $\bar m_{n+1}\in X_0$;
			\item[(ii)] $\supp(\bar m_{n+1}-\bar m_{n},\bar M_{n+1}-\bar M_{n})\subset[0,2^{-n}T)\times\T^d$;
			\item[(iii)] $\mathcal{D}(\bar m_{n+1},\bar m_{n})<2^{-n}$;
			\item[(iv)] $I(\bar m_{n+1})>I(\bar m_{n})+\beta_n$;
			\item[(v)] $\|(\bar m_{n+1}(t=0)-\bar m_{n}(t=0))*\eta_{\gamma_j}\|_{L^2}<2^{-n}$ for all $j\leq n$.
		\end{enumerate}
		From~(iii) we learn that the resulting sequence is Cauchy in $(X,\mathcal{D})$, hence there exists $\bar m\in X$ such that $\bar m_n\to\bar m$. This limit will still be continuous in $(0,T)\times\T^d$, since it is obtained as the limit of a sequence of continuous functions that, at every $t>0$, remains unchanged after finitely many steps, owing to~(ii).\\			
		Observe next that~(iv) immediately implies $I(\bar m_n)\nearrow 0$. If we can show strong convergence of $\bar m_n(t=0)$ to $\bar m(t=0)$ in $L^2(\T^d)$, we will be able to deduce $I(\bar m)=0$. To this end, note that, at $t=0$,
		\begin{equation}
			\|\bar m_n*\eta_{\gamma_n}-\bar m*\eta_{\gamma_n}\|_{L^2}\leq \sum_{j=0}^\infty\|\bar m_{n+j+1}*\eta_{\gamma_n}-\bar m_{n+j}*\eta_{\gamma_n}\|_{L^2}<\sum_{j=0}^\infty 2^{-n-j}=2^{-n+1},\label{telescope}
		\end{equation}  
		so that at $t=0$,
		\begin{equation*}
			\|\bar m_n-\bar m\|_{L^2}\leq \|\bar m_n-\bar m_n*\eta_{\gamma_n}\|_{L^2}+ \|\bar m_n*\eta_{\gamma_n}-\bar m*\eta_{\gamma_n}\|_{L^2}+\|\bar m-\bar m*\eta_{\gamma_n}\|_{L^2}\to 0
		\end{equation*}
		as $n\to\infty$, where we invoked~\eqref{gamma} and~\eqref{telescope}.\\			
		This shows that there exists a pair $(\bar m,\bar M)$ as required in the statement of Proposition~\ref{prop:proposition22}. However, replacing $2^{-n}$ by, say, $\frac1k 2^{-n}$ in~(iii) above, we even obtain a full sequence of such pairs that converge, as $k\to\infty$, to $m$ in $\mathcal{D}$, i.e.~in $CL^2_w$. This completes the proof.  
	\end{proof}
	We now come to the final step of the proof which uses the convex integration method to transition from subsolutions to weak solutions of (\ref{eq:euler}).
	\begin{prop}\label{prop:stepthree}
		Assume the lift $\tilde{\nu}$ of a measure-valued solution $\nu$ on $(0,T)\times \T^d$ with initial data $(\rho_0,m_0)\in L^{\infty}(\T^d)$ can be generated by a sequence $(\rho_j,m_j,M_j,Q_j)\in C((0,T)\times \T^d)$ of subsolutions with the properties:
		\begin{align*}
			(\rho_j,Q_j)&\in C([0,T]\times \T^d),\\
			m_j&\in C([0,T],L^2_{\operatorname{w}}(\T^d)),\\
			\|(\rho_j,m_j,M_j,Q_j)\|_{L^{\infty}}&\leq R,\\
			\rho_j&\geq \eta\text{ on }(0,T)\times \T^d,\\
			\frac{2}{d}e_{\operatorname{kin}}(\rho_j,m_j,M_j)+\rho_j^{\gamma}&< Q_j\text{ on }(0,T)\times 
			\T^d,\\
			e(\rho_j(0,x),m_j(0,x))&=\frac{d}{2}Q_j(0,x)\text{ for a.e.~}x\in\T^d,\\
			\int\limits_{\T^d}^{}Q_j(t,x)\dx&\leq \int\limits_{\T^d}^{}Q_j(0,x)\dx\text{ for all }t\in[0,T],\\
			\rho_j(0,\cdot)&\rightarrow \rho_0\text{ in }L^{\gamma}(\T^d),\\
			m_j(0,\cdot)&\rightarrow m_0\text{ in }L^{2}(\T^d)
		\end{align*}
		for some $R,\eta>0$.\\
		Then $\nu$ is generated by a uniformly bounded sequence of admissible weak solutions $(\bar{\rho}_j,\bar{m}_j)$ to (\ref{eq:euler}) such that
		\begin{align*}
			\|\bar{\rho}_j(t=0)-\rho_0\|_{L^{\gamma}(\T^d)}&\leq \frac{1}{j},\\
			\|\bar{m}_j(t=0)-m_0\|_{L^{2}(\T^d)}&\leq \frac{1}{j},\\
			\bar{\rho}_j&\geq \eta\text{ for a.e.~}(t,x)\in(0,T)\times \T^d.
		\end{align*}
	\end{prop}
	\begin{proof}
		Proposition \ref{theo:EJT} yields for every $j\in\N$ a sequence $\big(m_j^k\big)_{k\in\N}$ such that $(\rho_j,m_j^k)$ is a weak solution of (\ref{eq:euler}) on $(0,T)\times \T^d$, and $m_j^k\overset{k\rightarrow\infty}{\rightarrow} m_j$ in the $CL_{\operatorname{w}}^2$-topology, as well as $Q_j=\frac{2}{d}e(\rho_j,m_j^k)$ a.e.~on $(0,T)\times \T^d$. Moreover, $m_j^{k}(0,\cdot)=m_j(0,\cdot)$ for all $k$.\\
		We estimate
			\begin{align*}
				\frac{\big|m_j^k\big|^2}{d\rho_j}+\rho_j^{\gamma}=Q_j\geq \frac{2}{d}e_{\operatorname{kin}}(\rho_j,m_j,M_j)+\rho_j^{\gamma}\geq \frac{|m_j|^2}{d\rho_j}+\rho_j^{\gamma}.
			\end{align*}
			Hence, $\big|m_j^k \big|\geq |m_j|$ a.e.~on $(0,T)\times \T^d$. Therefore,
			\begin{align*}
				\left|\int\limits_{(0,T)}^{}\int\limits_{\T^d}^{}\big|m_j^{k}\big|^2-|m_j|^2\dx\dt\right|
				= & d\int\limits_{(0,T)}^{}\int\limits_{\T^d}^{}\rho_j\left(Q_j-\rho_j^{\gamma}-\frac{2}{d}e_{\operatorname{kin}}(\rho_j,m_j,M_j)\right)\dx\dt\\
				&+d\int\limits_{(0,T)}^{}\int\limits_{\T^d}^{}\rho_j\frac{2}{d}e_{\operatorname{kin}}(\rho_j,m_j,M_j)-\rho_j\frac{|m_j|^2}{d\rho_j}\dx\dt.
		\end{align*}
		It follows as in the proof of Proposition \ref{prop:steptwo} that
		\begin{align*}
			\left\|m_j^{k}\right\|_{L^2((0,T)\times\T^d)}- \|m_j\|_{L^2((0,T)\times \T^d)}\overset{k\rightarrow\infty}{\rightarrow} 0
		\end{align*}
		uniformly in $k$.\\
		Now choose a subsequence $k(j)$ satisfying
		\begin{align*}
			\underset{t\in(0,T)}{\sup}\int\limits_{\T^d}^{}m_j\cdot\left(m_j^{k(j)}-m_j\right)\dx\leq \frac{1}{j}.
		\end{align*}
		This yields
		\begin{align*}
			\int\limits_{(0,T)}^{}\int\limits_{\T^d}^{}\left|m_j^{k(j)}-m_j\right|^2\dx\dt&=\int\limits_{(0,T)}^{}\int\limits_{\T^d}^{}\left|m_j^{k(j)}\right|^2-|m_j|^2\dx\dt-2\int\limits_{(0,T)}^{}\int\limits_{\T^d}^{}m_j\cdot\left(m_j^{k(j)}-m_j\right)\dx\dt\\
			&\leq \left\|m_j^{k}\right\|_{L^2}- \|m_j\|_{L^2}+\frac{2}{j}\rightarrow 0.
		\end{align*}
		As $(\rho_j,m_j,M_j,Q_j)\overset{Y}{\rightharpoonup}\tilde{\nu}$, it is straightforward to check that $(\rho_j,m_j)\overset{Y}{\rightharpoonup}\nu$. Since $m_j^{k(j)}-m_j\overset{L^2}{\rightarrow} 0$, we therefore obtain that
		\begin{align*}
			(\bar{\rho}_j,\bar{m}_j):=\left(\rho_j,m_j^{k(j)}\right)
		\end{align*}
		generates $\nu$. Note also that
		\begin{align*}
			\left|\bar{m}_j\right|^2\leq dR\frac{\left|m_j^{k(j)}\right|^2}{d\rho_j}\leq dR(Q_j-\rho_j^{\gamma})
		\end{align*}
		almost everywhere on $(0,T)\times \T^d$. Thus, $\bar{m}_j$ is uniformly bounded.\\
		We have
		\begin{align*}
			\underset{t\in[0,T]}{\sup}\int\limits_{\T^d}^{}e\left(\bar{\rho}_j(t,x),\bar{m}_j(t,x)\right)\dx\leq \frac{d}{2}\int\limits_{\T^d}^{}Q_j(0,x)\dx=\int\limits_{\T^d}^{}e(\bar{\rho}_j(0,x),\bar{m}_j(0,x))\dx .
		\end{align*}
		Hence, $(\bar{\rho}_j,\bar{m}_j)$ is admissible for all $j$.\\
		Since $\bar{\rho}_j=\rho_j$ and $\bar{m}_j(t=0)=m_j(t=0)$, we immediately obtain
		\begin{align*}
			\bar{\rho}_j(t=0)\overset{L^{\gamma}}{\rightarrow}\rho_0\text{ and }\bar{m}_j(t=0)\overset{L^2}{\rightarrow}m_0.
		\end{align*}
		Choosing a subsequence in $j$, where we guarantee the above convergence to be faster than $\frac{1}{j}$, finishes the proof.
	\end{proof}
	A subsequent application of Proposition \ref{theo:stepone}, Proposition \ref{prop:steptwo}, and Proposition \ref{prop:stepthree} yields the assertion of Theorem \ref{theo:mainresult}. Ignoring the energy in the previous proofs Theorem \ref{theo:weakmainresult} follows by the very same reasoning. Note that the proof is then significantly simpler, especially Proposition \ref{prop:steptwo} is not needed.
\section{The case of measure-valued solutions consisting of two Dirac measures}\label{sect:application}
In this section we want to construct a class of non-atomic measure-valued solutions to which Theorem~\ref{theo:weakmainresult} can be applied. These measures will be a convex combination of two Dirac measures supported on weak solutions. As shown in~\cite{CFKW17} or~\cite{GW20} such measures are in general not generated by weak solutions. In these articles the fact that the involved states are not wave-cone connected plays an important role. However, we will show that if the underlying states are wave-cone connected, the considered measure-valued solution is, under some further assumptions, generated by weak solutions.\\
More concretely, we obtain the following theorem.
\begin{theo}\label{theo:applicationnonatomic}
	Let $T>0$. Let $(\rho_1,m_1),(\rho_2,m_2)\in L^{\infty}((0,T)\times \T^d)$ be two distinct bounded weak solutions of (\ref{eq:euler}) on $(0,T)\times \T^d$ with initial data $(\rho_0^1,m^1_0),(\rho^2_0,m^2_0)\in L^{\infty}(\T^d)$, respectively, and let $\lambda\in(0,1)$ satisfy the following conditions:
	\begin{itemize}
		\item There exists some $\eta>0$ such that $\rho_1,\rho_2\geq \eta$.
		\item The corresponding lifted states $z_1:=\Theta(\rho_1,m_1)$ and $z_2:=\Theta(\rho_2,m_2)$ satisfy $z_1-z_2\in \underset{\omega\in \mathbb{S}^d}{\bigcup}\image\mathbb{B}_E(\omega)$ a.e.~in $(0,T)\times \T^d$.
		\item There exists some $w\in W^{2,\infty}((0,T)\times \T^d)$ and some $\sigma\in C([0,T]\times \T^d)$ such that $\lambda z_1+(1-\lambda)z_2=\sigma+\mathcal{B}_Ew$.
	\end{itemize}
	Then the measure-valued solution $\nu=\lambda\delta_{(\rho_1,m_1)}+(1-\lambda)\delta_{(\rho_2,m_2)}$ can be generated by a uniformly bounded sequence of weak solutions.
\end{theo}
\begin{proof}
	As $(\rho_1,m_1)$ and $(\rho_2,m_2)$ are weak solutions, it is immediate that $\nu$ is a measure-valued solution with initial data $\lambda(\rho^1_0,m^1_0)+(1-\lambda)(\rho^2_0,m^2_0)$. So it suffices to check the conditions on $\nu$ required to apply Theorem~\ref{theo:weakmainresult}.\\
	First, note that if $(\rho,m,M,Q)\in \supp(\tilde{\nu}_{(t,x)})$ then $\rho=\rho_j\geq \eta$ with $j\in\{1,2\}$.\\
	Next, we claim that for $R:=C_{\mathcal{B}}\|z_1-z_2\|_{L^{\infty}}$, where $C_{\mathcal{B}}$ is the constant from Lemma \ref{lem:waveconeconvexity}, it holds that the lift $\tilde{\nu}$ satisfies $\langle\tilde{\nu}_{(t,x)},f\rangle\geq Q^{R}_{\mathcal{B}_E}f(\langle\tilde{\nu}_{(t,x)},\operatorname{id}\rangle)$ for a.e.~$(t,x)\in (0,T)\times\T^d$ and all $f$ continuous:\\
	Let $f\in C(\R^m)$. By assumption we have that $z_1-z_2\in \image\mathbb{B}_E(\omega)$ almost everywhere for some $\omega\in \mathbb{S}^d$. Then Lemma \ref{lem:waveconeconvexity} implies
	\begin{align*}
		\langle \tilde{\nu},f\rangle
		=\lambda f(z_1)+(1-\lambda)f(z_2) 
		\geq Q_{\mathcal{B}}^{C_{\mathcal{B}}\|z_1-z_2\|_{L^{\infty}}}f(\langle\tilde{\nu},\operatorname{id}\rangle).
	\end{align*}
	The remaining condition that $\langle\tilde{\nu},\operatorname{id}\rangle=\lambda z_1+(1-\lambda)z_2=\sigma+\mathcal{B}_Ew$ for some $w\in W^{2,\infty}((0,T)\times \T^d)$ and some $\sigma\in C([0,T]\times \T^d)$ is fulfilled by assumption.
\end{proof}
\begin{rem}
	Here, the choice of initial data for the diatomic measure $\nu$ above is somewhat artificial, since in the energy admissible case the appropriate choice would be $\lambda\delta_{(\rho_0^1,m_0^1)}+(1-\lambda)\delta_{(\rho_0^2,m_0^2)}$. This however only corresponds to a function if the initial data coincides.
\end{rem}
We want to give an explicit example of two weak solutions $(\rho_1,m_1)$ and $(\rho_2,m_2)$ that fulfill the conditions in Theorem \ref{theo:applicationnonatomic}.
\begin{example}\label{ex:exampleweakmaintheorem}
	Consider the case $d=2$ and fix $T>0$ arbitrary. Set $\rho_1=\rho_2=1$. Let $\alpha,\beta\in L^{\infty}(\T^1,\R)$ and define
	\begin{align*}
		m_1(t,x,y):=\begin{pmatrix}
			\alpha(y)\\
			0
		\end{pmatrix}\text{ and }m_2(t,x,y):=\begin{pmatrix}
		\beta(y)\\
		0
	\end{pmatrix}.
	\end{align*}
	Then clearly $(\rho_1,m_1)$ and $(\rho_2,m_2)$ are weak solutions of (\ref{eq:euler}) with initial data $(\rho^1_0,m^1_0)=(1,\alpha,0)$ and $(\rho^2_0,m^2_0)=(1,\beta,0)$. As the densities are equal, Remark 4.9 in \cite{GW20} yields that the corresponding lifted states $z_1,z_2$ are wave-cone connected. Observe that Lemma \ref{lem:eulerpotentialproperties} below in combination with Proposition \ref{prop:potentialcharacterization} yields that $z_1-z_2\in \underset{\omega\in\mathbb{S}^d}{\bigcup}\image\mathbb{B}_E(\omega)$. In this example we will use the explicit form of the potential $\mathcal{B}_E$, which is given in (\ref{eq:explicitformofpotential}) below for the case of two space dimensions.\\
	Now let $\lambda\in (0,1)$. It remains to check that
	\begin{align*}
		\lambda z_1+(1-\lambda)z_2=\left(1,\lambda \alpha+(1-\lambda)\beta,0,\frac{\lambda}{2} \alpha^2+\frac{1-\lambda}{2}\beta^2,0,1+\frac{\lambda}{2} \alpha^2+\frac{1-\lambda}{2}\beta^2\right)^{\operatorname{t}}
	\end{align*}
	is of the form $\sigma+\mathcal{B}_Ew$ for some $\sigma\in C([0,T]\times \T^2,\R^6)$ and some $w\in W^{2,\infty}((0,T)\times \T^2,\R^9)$.\\
	For that observe that $\psi\colon \R\times \R^2\times S_0^2\times \R\to \R^6,\ (\rho,m,M,Q)\mapsto (\rho,m,M+Q\mathbb{E}_2)$ is an isomorphism. Define
	\begin{align*}
		\sigma(t,x,y)=(\sigma_{\rho},\sigma_{m_1},\sigma_{m_2},\sigma_{M_{11}},\sigma_{M_{12}},\sigma_Q)(t,x,y):=\int\limits_{0}^{1}(\lambda z_1+(1-\lambda)z_2)(t,x,\tilde{y})\dd\tilde{y}.
	\end{align*}
	Then
	\begin{align*}
		\psi(\lambda z_1+(1-\lambda)z_2-\sigma)=\left(0,\lambda \alpha+(1-\lambda)\beta-\sigma_{m_1},0,\lambda \alpha^2+(1-\lambda)\beta^2-\sigma_{M_{11}}-\sigma_Q,0,0 \right)^{\operatorname{t}}.
	\end{align*}
	From the proof of Lemma \ref{lem:eulerpotentialproperties} below we obtain that the equation
	\begin{align*}
		\lambda z_1+(1-\lambda)z_2-\sigma=\mathcal{B}_Ew
	\end{align*}
	is satisfied for some function $w$ if and only if
	\begin{align}
		\psi(\lambda z_1+(1-\lambda)z_2-\sigma)=\begin{pmatrix}
			0&0&0&0&0&0&0&\partial^2_{y}&0\\
			0&0&0&0&\frac{1}{2}\partial^2_{y}&0&0&0&\frac{1}{2}\partial^2_{y}\\
			0&0&0&0&0&0&0&0&0\\
			0&0&0&0&0&\partial^2_{y}&0&0&0\\
			0&0&0&0&0&0&0&0&0\\
			0&0&0&0&0&0&0&0&0
		\end{pmatrix}\cdot w\label{eq:examplepotential}
	\end{align}
	since $\lambda z_1+(1-\lambda)z_2-\sigma$ is only a function of $y$. Define $a:=\lambda \alpha+(1-\lambda)\beta-\sigma_{m_1}$ and $b:=\lambda \alpha^2+(1-\lambda)\beta^2-\sigma_{M_{11}}-\sigma_Q$. Then $a,b$ have zero average. Thus, the antiderivative $\tilde{f}_a(z):=\int\limits_{0}^{z}a(z')\dz'$ is $1$-periodic and continuous. Hence, there exists some $c\in [0,1]$ such that $\tilde{f}_a(c)=\int\limits_{0}^{1}\tilde{f}_a(z)\dz$. Therefore, $f_a(z):=\tilde{f}_a(z)-\tilde{f}_a(c)=\int\limits_{c}^{z}a(z')\dz'$ is periodic and has zero average. Thus, the second antiderivative $F_a(y):=\int\limits_{0}^{y}f_a(z)\dz$ of $a(\cdot)$ is periodic. Moreover, $F_a$ is clearly bounded. Analogously, we pick a periodic second antiderivative $F_b$ of $b$. The choice
	\begin{align*}
		w(t,x,y)=(0,0,0,0,F_a(y),F_b(y),0,0,F_a(y))^{\operatorname{t}}
	\end{align*}
	now satisfies the equation (\ref{eq:examplepotential}) above and lies in $W^{2,\infty}((0,T)\times \T^2)$, hence it is a potential function with the desired properties.\\
	Note that the above example contains also the case of shear flows, since we only assumed $\alpha,\beta\in L^{\infty}$.
\end{example}
We can also find a more concrete setting for admissible measure-valued solutions consisting of two Dirac measures. The following result gives conditions for the application of Theorem \ref{theo:mainresult} in this case.
\begin{theo}\label{theo:admissiblenonatomic}
	Let $T>0$. Let $(\rho_1,m_1),(\rho_2,m_2)\in L^{\infty}((0,T)\times \T^d)$ be two distinct bounded admissible weak solutions of (\ref{eq:euler}) on $(0,T)\times \T^d$ with the same initial data $(\rho_0,m_0)\in L^{\infty}(\T^d)$ and let $\lambda\in(0,1)$ satisfy the following conditions:
	\begin{itemize}
		\item There exists some $\eta>0$ such that $\rho_1,\rho_2\geq \eta$.
		\item The energy profiles of $(\rho_1,m_1)$ and $(\rho_2,m_2)$ are equal, i.e.~$e(\rho_1,m_1)=\varepsilon=e(\rho_2,m_2)$ a.e.~on $(0,T)\times\T^d$. Moreover, the joint energy profile $\varepsilon$ is continuous on $[0,T]\times \T^d$.
		\item The corresponding lifted states $z_1:=\Theta(\rho_1,m_1)$ and $z_2:=\Theta(\rho_2,m_2)$ satisfy $z_1-z_2\in \underset{\omega\in \mathbb{S}^d}{\bigcup}\image\mathbb{B}_E(\omega)$ a.e.~in $(0,T)\times \T^d$.
		\item There exists some $w\in W^{2,\infty}((0,T)\times \T^d)$ and some $\sigma\in C([0,T]\times \T^d)$ such that $\lambda z_1+(1-\lambda)z_2=\sigma+\mathcal{B}_Ew$.
	\end{itemize}
	Then the admissible measure-valued solution $\nu=\lambda\delta_{(\rho_1,m_1)}+(1-\lambda)\delta_{(\rho_2,m_2)}$ can be generated by a uniformly bounded sequence of admissible weak solutions $(\rho_j,m_j)$ such that
	\begin{align*}
		\|\rho_j(t=0)-\rho_0\|_{L^{\gamma}(\T^d)}&\leq \frac{1}{j},\\
		\|m_j(t=0)-m_0\|_{L^{2}(\T^d)}&\leq \frac{1}{j},\\
		\rho_j&\geq \tilde{\eta}\text{ a.e.~on }(0,T)\times \T^d
	\end{align*}
	for some $\tilde{\eta}>0$.
\end{theo}
\begin{proof}
	Since $(\rho_1,m_1)$ and $(\rho_2,m_2)$ are admissible weak solutions, a direct computation shows that $\nu$ is an admissible measure-valued solution with initial data $(\rho_0,m_0)$.\\
	Exactly as in the proof of Theorem \ref{theo:applicationnonatomic} we can show that $\rho\geq \eta$ for all $(\rho,m,M,Q)\in \supp(\tilde{\nu})$ and that there exists some $R>0$ such that $\langle\tilde{\nu},f\rangle\geq Q_{\mathcal{B}_E}^Rf(\langle\tilde{\nu},f\rangle)$ a.e.~on $(0,T)\times \T^d$ for all $f\in C(\R^m)$. Moreover, the condition $\langle\tilde{\nu},\operatorname{id}\rangle=\sigma+\mathcal{B}_Ew$ for some $\sigma\in C([0,T]\times \T^d)$ and some $w\in W^{2,\infty}((0,T)\times\T^d)$ holds by assumption.\\
	In order to be able to apply Theorem \ref{theo:mainresult} we observe that at a.e.~point in $(0,T)\times \T^d$ for all $(\rho,m,M,Q)\in\supp(\tilde{\nu}_{(t,x)})$ it holds that
	\begin{align*}
		Q&=\frac{|m_1|^2}{d\rho_1}+\rho_1^{\gamma}=\frac{|m_2|^2}{d\rho_2}+\rho_2^{\gamma}\\
		&=\lambda\frac{2}{d}e(\rho_1,m_1)+(1-\lambda)\frac{2}{d}e(\rho_2,m_2)=\frac{2}{d}\langle\nu_{},e\rangle.
	\end{align*}
	We also immediately see that $\langle\nu_{},e\rangle=\varepsilon$ is continuous in $(t,x)$.
\end{proof}
\begin{rem}
	In the case $d=2$ one can guarantee the existence of two admissible weak solutions $(\rho_1,m_1),(\rho_2,m_2)$ that fulfill all conditions in Theorem \ref{theo:admissiblenonatomic} except possibly for the last one. This can be achieved by the method of convex integration. For that choose $\rho_0= \eta$ for some $\eta>0$ and fix some time $T>0$. Our general assumptions imply $\gamma=1+\frac{2}{d}=2$ and $p(\rho)=\rho^2$. Then Theorem 2.1 in \cite{C14} states that there exists some bounded initial momentum $m_0$ such that there are infinitely many weak solutions $(\rho,m)$ satisfying $\rho=\rho_0=\eta$ and $|m|^2=\eta\cdot \chi$ for some smooth function in time $\chi$. Also $|m_0|^2=\eta\cdot \chi(0)$ holds. Note that $m$ is then clearly bounded. The energy profile is $\varepsilon=\frac{1}{2}\chi+\eta^2$ and thus is continuous. Note that Theorem 2.2 in \cite{C14} implies that one can choose the function $\chi$ such that the weak solutions $(\rho,m)$ above are admissible, since the notion of energy admissibility in \cite{C14} is stronger than ours. The fact that the densities of the so obtained weak solutions are equal implies that the corresponding lifted states are wave-cone connected almost everywhere by using Remark 4.10 in \cite{GW20}.\\
	It is not clear if by this convex integration procedure the solutions can be constructed such that a convex combination of their lifts is equal to $\mathcal{B}_Ew$ modulo some uniformly continuous function where $w\in W^{2,\infty}((0,T)\times \T^d)$. Although a rigorous proof of this assertion is not yet available, we assert that this is presumably correct.
\end{rem}
It is not hard, but rather lengthy to determine the relaxed Euler potential $\mathcal{B}_E$ from Proposition 4.1 in \cite{G20} explicitly. However, in the case $d=2$ this is manageable and so we conclude this section by showing some basic properties of $\mathcal{B}_E$ used above for finding an explicit application of our main results.
\begin{lem}\label{lem:eulerpotentialproperties}
	Let $d=2$. Then $\mathcal{B}_E$ is a potential for $\mathcal{A}_E$ in the sense of (\ref{eq:potentialdefinition}) and has constant rank.
\end{lem}
\begin{proof}
	We first show that $\mathcal{B}_E$ satisfies (\ref{eq:potentialdefinition}):\\
	Suppose that $z=(\rho,m,M,Q)\in C^{\infty}(\T^{2+1})$ is a given periodic and average-free vector field such that $\mathcal{A}_Ez=0$. The potential $\mathcal{B}_E$ has been constructed in Proposition 4.1 in \cite{G20}. This construction consists of the following four steps. First apply the linear isomorphism already mentioned in Example \ref{ex:exampleweakmaintheorem}
	\begin{align}\label{eq:psidefinition}
		\psi\colon \R\times \R^2\times S_0^2\times \R\to S^3,\ (\rho,m,M,Q)\mapsto (\rho,m,M+Q\mathbb{E}_2).
	\end{align}
	Then use a variant of Poincar\'e's lemma on the divergence-free vector-field $U=\psi(z)$. Next apply a certain bijective linear transformation $\chi$ and finally Poincar\'e's lemma is used one more time.\\
	Note that in the case $d=2$ the step including Poincar\'e's lemma boils down to finding a vector field $u\in C^{\infty}(\T^{2+1},\R^3)$ such that $\curl u=f$ if $f\in C^{\infty}(\T^{2+1},\R^3)$ satisfies $\operatorname{div}f=0$. Since $z$ lives on the torus and is average-free, the divergence-free matrix-field $U=\psi(z)$ has the same properties. Define for all $x\in\T^3$
	\begin{align*}
		w_j(x):=\sum\limits_{k\in\Z^3\backslash\{0\}}^{}i\frac{k\times \hat{U}_j(k)}{|k|^2}e^{2\pi ixk},
	\end{align*}
	where $U_j$ is the vector consisting of the $j$-th row of $U$, $j=1,2,3$. Since $U_j$ is smooth, so is $w_j$. Also the periodicity is clear. Moreover, because $U_j$ is divergence-free and average-free, we obtain for all $x\in \T^3$
	\begin{align*}
		\curl w_j(x)=-\sum\limits_{k\neq 0}^{}\frac{k\times(k\times \hat{U}_j(k))}{|k|^2}e^{2\pi ixk}=\sum\limits_{k\neq 0}^{}\hat{U}_j(k)e^{2\pi ixk}=U_j(x).
	\end{align*}
	It is also straightforward to check that $w_j$ is again average-free.\\
	Now the subsequent linear isomorphism $\chi$ preserves again average-freeness and periodicity. An analogous discussion as above implies that the second usage of this variant of Poincar\'e's lemma yields again periodic smooth vector-fields. Vice versa, from the construction and the fact that images of the $\operatorname{curl}$-operator are $\operatorname{div}$-free we get that $\mathcal{A}\mathcal{B}u=0$ for all $u\in C^{\infty}(\T^3)$. Therefore, using Proposition \ref{prop:potentialcharacterization} we can infer that $\mathcal{B}_E$ is a potential in the sense of (\ref{eq:potentialdefinition}) if $\mathcal{B}_E$ and $\mathcal{A}_E$ are constant rank operators. We already know that $\mathcal{A}_E$ has constant rank, see Remark~\ref{rem:constantrank}. In order to check that $\mathcal{B}_E$ also has constant rank we make the following observations:\\
	If one carries out the four steps in the construction of $\mathcal{B}_E$ for $d=2$ explicitly, the resulting operator is
	\begin{align}
		\mathcal{B}_E=\psi^{-1}\circ\begin{pmatrix}
			\partial_x^2 & \partial_x\partial_y & 0 & 0 & 0 & 0& \partial_x\partial_y & \partial_y^2 & 0\\
			-\partial_t\partial_x & -\frac{1}{2}\partial_t\partial_y & \frac{1}{2}\partial_x\partial_y & \frac{1}{2}\partial_x\partial_y & \frac{1}{2}\partial_y^2 & 0 & -\frac{1}{2}\partial_t\partial_y & 0 & \frac{1}{2}\partial_y^2\\
			0 & -\frac{1}{2}\partial_t\partial_x & -\frac{1}{2}\partial_x^2 & -\frac{1}{2}\partial_x^2 & -\frac{1}{2}\partial_x\partial_y & 0 & -\frac{1}{2}\partial_t\partial_x & -\partial_t\partial_y & -\frac{1}{2}\partial_x\partial_y\\
			\partial_t^2 & 0 & -\partial_t\partial_y & -\partial_t\partial_y & 0 & \partial_y^2 & 0 & 0 & 0\\
			0 & \frac{1}{2}\partial_t^2 & \frac{1}{2}\partial_t\partial_x & \frac{1}{2}\partial_t\partial_x & -\frac{1}{2}\partial_t\partial_y & -\partial_x\partial_y & \frac{1}{2}\partial_t^2 & 0 &-\frac{1}{2}\partial_t\partial_y\\
			0 & 0 & 0 & 0 & \partial_t\partial_x & \partial_x^2 & 0 & \partial_t^2 & \partial_t\partial_x
		\end{pmatrix}.\label{eq:explicitformofpotential}
	\end{align}
	Let $\omega\in \mathbb{S}^2$ be such that all three entries $\omega_t,\omega_x,\omega_y$ are non-zero. Then a straightforward linear algebra calculation shows that the matrix $\mathbb{B}_E(\omega)$ can be transformed by elementary row operations into
	\begin{align*}
		\begin{pmatrix}
			\frac{\omega_x}{\omega_y} & 1 & 0 & 0 & 0 & 0 & 1 & \frac{\omega_y}{\omega_x} & 0\\
			0 & 0 & 0 & 0 & 0 & 0 & 0 & 0 & 0\\
			0 & 0 & 0 & 0 & 0 & 0 & 0 & 0 & 0\\
			-\frac{\omega_t}{\omega_y} & 0 & 1 & 1 & 0 & -\frac{\omega_y}{\omega_t} & 0 & 0 & 0\\
			0 & 0 & 0 & 0 & 0 & 0 & 0 & 0 & 0\\
			0 & 0 & 0 & 0 & 1 & \frac{\omega_x}{\omega_t} & 0 & \frac{\omega_t}{\omega_x} & 1\\
		\end{pmatrix},
	\end{align*}
	which clearly has rank equal to three. In the case of one or two entries of $\omega$ being zero we also obtain $\rank(\mathbb{B}_E(\omega))=3$, which is an even simpler observation. Altogether, we conclude that $\mathcal{B}_E$ has constant rank equal to three. 
\end{proof}
For the generalization of the previous lemma to higher dimensions, the only difficulty is to explicitly determine the form of $\mathcal{B}_E$ and show that it has constant rank. This, however, turns out to be a very tedious calculation and so a rigorous proof is still missing, cf.~also Remark~\ref{rem:constantrank}.
\section{Necessary Conditions}\label{sect:necessaryconditions}
The last section is dedicated to finding necessary conditions that are fulfilled by measure-valued solutions which are generated by weak solutions. There is already a result on this topic proved by Chiodaroli et al., cf.~Theorem 3 in \cite{CFKW17}. We will adjust the necessary conditions given there to our situation and gain some improvements due to the uniform boundedness of our generating sequence. Actually, we will also consider the case of vanishing viscosity sequences as generating sequences. Including vanishing viscosity limits in Theorem \ref{theo:weakmainresult} and Theorem \ref{theo:mainresult} unfortunately cannot be done in an obvious way and might be a subject for future work. 
\begin{theo}\label{theo:necessaryconditions}
	Let $T>0$. Suppose $\nu$ is a Young measure which is generated by a sequence of functions $(\rho_n,m_n)$ over $(0,T)\times \T^d$ satisfying:
	\begin{itemize}
		\item There exists $M>0$ such that $\|(\rho_n,m_n)\|_{L^{\infty}((0,T)\times \T^d)}\leq M$ for all $n\in\N$.
		\item The sequence $(\rho_n,m_n)$ is a vanishing viscosity sequence with the property that $\mu_n\left\|\nabla_x \frac{m_n}{\rho_n}\right\|_{L^2((0,T)\times \T^d)}\rightarrow 0$ or consists of weak solutions of (\ref{eq:euler}).
		\item The initial values converge weakly, i.e.~$(\rho_n^0,m_n^0)\rightharpoonup (\rho_0,m_0)$ in $(L^{\gamma}\times L^2)(\T^d)$ for some $(\rho_0,m_0)\in (L^{\gamma}\times L^2)(\T^d)$.
		\item There exists some $\eta>0$ such that the densities satisfy $\rho_n\geq \eta$ for all $n\in\N$.
	\end{itemize}
	Then $\nu$ is a measure-valued solution of (\ref{eq:euler}) on $(0,T)\times \T^d$ with initial data $(\rho_0,m_0)\in L^{\infty}(\T^d)$ satisfying the properties:
	\begin{itemize}
		\item For a.e.~$(t,x)\in(0,T)\times \T^d$ it holds that $\supp\left(\tilde{\nu}_{(t,x)} \right)\subset \{(\rho,m,M,Q)\,:\,\rho\geq \eta \}\cap B_R(0)$ for some $R>0$.
		\item The barycenter of the lift satisfies $\langle\tilde{\nu},\operatorname{id}\rangle\in L^{\infty}((0,T)\times \T^d)$ and is $\mathcal{A}_E$-free. 
		\item For a.e.~$(t,x)\in(0,T)\times \T^d$ and all $f\in C(\R^+\times \R^d\times S_0^d\times \R^+)$ with $|f(z)|\leq C(1+|z|^2)$ for some $C>0$ it holds that $\langle\tilde{\nu}_{(t,x)},f\rangle\geq Q_{\mathcal{A}_E}f(\langle\tilde{\nu}_{(t,x)},\operatorname{id}\rangle)$.
	\end{itemize}
	If we additionally assume that $(\rho_n,m_n)$ is a sequence of admissible weak solutions or an energy admissible vanishing viscosity sequence with $\|(\rho_n^0,m_n^0)-(\rho^0,m^0)\|_{L^{\gamma}\times L^2}\rightarrow 0$, then $\nu$ is an admissible measure-valued solution of (\ref{eq:euler}) with initial data $(\rho^0,m^0)$.
\end{theo}
\begin{rem}
	Note that the second assumption $\mu_n\left\|\nabla_x \frac{m_n}{\rho_n}\right\|_{L^2}\rightarrow 0$ is automatic in the case that $(\rho_n,m_n)$ is an vanishing viscosity sequence satisfying the energy inequality.
\end{rem}
\begin{rem}
	Here, we introduced the \textrm{quasiconvex envelope} $Q_{\mathcal{A}_E}f$ of a function $f$. This is defined for a continuous function $f$ as
	\begin{align*}
		Q_{\mathcal{A}_E}f(z)=\inf\left\{\int\limits_{\T^{d+1}}^{}f(z+w(t,x))\dx\dt\,:\, w\in C^{\infty}(\T^{d+1})\cap \ker\mathcal{A}_E,\ \underset{\T^{d+1}}{\dashint}w\,\dx\dt=0 \right\}
	\end{align*}
	for all $z\in \R^+\times \R^d\times S_0^d\times \R^+$. Note that this definition can be found in a more general setting in \cite{FM99}. Due to the specific properties of $\mathcal{B}_E$ we investigated in the case $d=2$ in Section \ref{sect:application}, we can actually formulate this in terms of the potential $\mathcal{B}_E$ using Corollary 5 in \cite{R18} by
	\begin{align*}
		Q_{\mathcal{A}_E}f(z)=Q_{\mathcal{B}_E}f(z):=\inf\left\{ \int\limits_{\mathcal{Q}}^{}f(z+\mathcal{B}_Ew(t,x))\dx\dt\,:\, w\in C_c^{\infty}(\mathcal{Q}) \right\}.
	\end{align*}
	This is the reason for us to call $Q_{\mathcal{B}_E}^qf$ the truncated quasiconvex envelope.
\end{rem}
\begin{proof}[Proof of Theorem \ref{theo:necessaryconditions}]
	Suppose $(\rho_n,m_n)$ is a vanishing viscosity sequence generating $\nu$. The case of a generating sequence consisting of weak solutions is simpler and follows by basically the same arguments as below, hence this case is omitted.\\
	Since $(\rho_n,m_n)\overset{Y}{\rightharpoonup}\nu$, it holds that $z_n:=\Theta(\rho_n,m_n)\overset{Y}{\rightharpoonup}\tilde{\nu}$. As $\rho_n$ is bounded away from zero by $\eta>0$, we infer that $\|z_n\|_{L^{\infty}((0,T)\times \T^d)}\leq R$ for some $R>0$. Thus, standard Young measure theory implies that
	\begin{align*}
		\supp(\tilde{\nu}_{(t,x)})\subset \{(\rho,m,M,Q)\,:\, \rho \geq \eta\}\cap B_R(0)
	\end{align*}
	for a.e.~$(t,x)\in (0,T)\times \T^d$. Moreover, we have $z_n\overset{*}{\rightharpoonup}\langle\tilde{\nu},\operatorname{id}\rangle$ in $L^{\infty}((0,T)\times \T^d)$. This shows the first two bullet points.\\
	In the proof of Theorem 4.10 in \cite{GW20} it is shown that the vanishing viscosity sequence $(z_n)$ satisfies
	\begin{align}
		\mathcal{A}_Ez_n\rightarrow 0\text{ in }W^{-1,2}((0,T)\times \T^d).\label{eq:vanishingviscosityconvergence}
	\end{align}
	For the reader's convenience we also give the proof here:\\
	Let us denote by $\mu_n$ the sequence of viscosity constants tending to zero. Then
	\begin{align*}
		\|\mathcal{A}_Ez_n\|_{W^{-1,2}((0,T)\times \T^d)}&=\underset{\begin{subarray}{c}
				\|\varphi\|_{W^{1,2}_0((0,T)\times \T^d)}\leq 1, \\
				\varphi\in C_c^{\infty}((0,T)\times \T^d)
		\end{subarray}}{\sup}\left|\left\langle \mu_n \operatorname{div}\mathbb{S}\left(\nabla \frac{m_n}{\rho_n}\right),\varphi\right\rangle\right|\\
		&=\underset{\begin{subarray}{c}
				\|\varphi\|_{W^{1,2}_0((0,T)\times \T^d)}\leq 1, \\
				\varphi\in C_c^{\infty}((0,T)\times \T^d)
		\end{subarray}}{\sup}\left|\int\limits_{(0,T)\times \T^d}^{}\mu_n\mathbb{S}\left(\nabla \frac{m_n}{\rho_n}\right):\nabla\varphi\dx\dt \right|\\
		&\leq C\mu_n\left\|\nabla_x \frac{m_n}{\rho_n}\right\|_{L^2((0,T)\times \T^d)}\\
		&\rightarrow 0.
	\end{align*}
	Here $\mathbb{S}(\nabla \cdot)$ denotes viscosity stress tensor from the compressible Navier-Stokes equations. This shows (\ref{eq:vanishingviscosityconvergence}).\\
	Combining $z_n\overset{*}{\rightharpoonup}\langle \tilde{\nu},\operatorname{id}\rangle$ and (\ref{eq:vanishingviscosityconvergence}) yields that $\mathcal{A}_E\langle\tilde{\nu},\operatorname{id}\rangle=0$. The support properties of $\tilde{\nu}$ and Lemma 2.1 in \cite{EJT20} then imply that $\nu$ is a measure-valued solution of (\ref{eq:euler}) with initial data $(\rho_0,m_0)$. Note that $(\rho_0,m_0)\in L^{\infty}(\T^d)$ due to the uniform boundedness of $(\rho_n,m_n)$.\\
	Now Proposition 3.8 in \cite{FM99} yields that
	\begin{align*}
		\langle\tilde{\nu}_{(t,x)},f\rangle\geq Q_{\mathcal{A}_E}f(\langle\tilde{\nu}_{(t,x)},\operatorname{id}\rangle)
	\end{align*}
	for a.e.~$(t,x)\in (0,T)\times \T^d$ and for all $f\in C(\R^m)$ with $|f(z)|\leq C(1+|z|^2)$ for some $C>0$.
	\\
	\\
	We now assume that the functions $(\rho_n,m_n)$ are energy admissible with $\|(\rho_n^0,m_n^0)-(\rho^0,m^0)\|_{L^{\gamma}\times L^2}\rightarrow 0$. Note that in the case that $(\rho_n,m_n)$ is a vanishing viscosity sequence, we have that $\sqrt{\mu_n}\left\|\nabla_x \frac{m_n}{\rho_n}\right\|_{L^2((0,T)\times \T^d)}\leq \tilde{M}$ for some $\tilde{M}>0$, which follows from the energy inequality and Korn's inequality. Thus, $\mu_n\left\|\nabla_x \frac{m_n}{\rho_n}\right\|_{L^2}\rightarrow 0$.\\
	It only remains to check the energy inequality for $\nu$. 
	The energy density function $e$ is continuous on $[\eta,\infty)\times \R^{d}$. Thus,
	\begin{align}
		e(\rho_n,m_n)\rightharpoonup \langle\nu,e\rangle\text{ in }L^1((0,T)\times \T^d).\label{eq:testingtheinequality}
	\end{align}
	Since $\rho_n\geq \eta$ and $\rho_n\in CL^{\gamma}_{\operatorname{w}}$, it holds that $\rho_n^0\geq \eta$. Therefore, the convergence of $\rho_n^0$ and $m_n^0$ implies
	\begin{align*}
		\int\limits_{\T^d}^{}e(\rho_n^0,m_n^0)\dx\rightarrow \int\limits_{\T^d}^{}e(\rho^0,m^0)\dx.
	\end{align*}
	This together with (\ref{eq:testingtheinequality}) yields the energy inequality for $\nu$ at a.e.~time $t$, which finishes the proof.
\end{proof}
\begin{rem}\label{rem:twodimensionsnecessary}
	Without loss of generality set $T=1$. For general $T>0$ use a reparametrization.\\
	We investigated the potential $\mathcal{B}_E$ for $d=2$ more deeply in the previous section. In particular, we obtained that $\mathcal{B}_E$ is a constant rank potential for $\mathcal{A}_E$ in the sense of (\ref{eq:potentialdefinition}). Let $1<p<\infty$. By a Calder\'on-Zygmund argument there exists a constant $C_p>0$ such that for all $z\in L^p(\T^3)=L^p((0,1)\times \T^2)$ with $\mathcal{A}_Ez=0$ and $\int\limits_{\T^3}^{}z\,\dx\dt=0$ there exists some potential $u$, i.e.~$\mathcal{B}_Eu=z$, such that the inequality
	\begin{align*}
		\|u\|_{W^{2,p}(\T^3)}\leq C_p\|z\|_{L^p(\T^3)}
	\end{align*}
	holds, cf.~Theorem A.2 in \cite{SkW21}.\\
	With this at hand we can further specify the form of $\langle\tilde{\nu},\operatorname{id}\rangle$ in Theorem \ref{theo:necessaryconditions}:\\
	Since $\langle\tilde{\nu},\operatorname{id}\rangle\in L^{\infty}((0,1)\times\T^2)\subset L^p((0,1)\times \T^2)$ is $\mathcal{A}_E$-free, the above argumentation yields for all $1<p<\infty$ some potential $w_p\in W^{2,p}(\T^3)$ satisfying
	\begin{align*}
		\langle\tilde{\nu},\operatorname{id}\rangle=\int\limits_{0}^{1}\int\limits_{\T^2}^{}\langle\tilde{\nu},\operatorname{id}\rangle\dx\dt+\mathcal{B}_Ew_p.
	\end{align*}
	Note that $W^{2,p}(\T^3)\subset W^{2,p}((0,1)\times \T^2)$. Thus, $w_p\in W^{2,p}((0,1)\times \T^2)$ and we can identify $\sigma=\int\limits_{0}^{1}\int\limits_{\T^2}^{}\left\langle\Theta_{\sharp}\nu,\operatorname{id}\right\rangle\dx\dt$ to be actually constant.
\end{rem}
To conclude this paper let us compare the sufficient conditions given in Theorems~\ref{theo:weakmainresult} and \ref{theo:mainresult} for a measure-valued solution to be generated by weak solutions with the necessary conditions from Theorem \ref{theo:necessaryconditions}.
\begin{rem}\label{rem:comparison}
	From our sufficient conditions, we are only able to decide for compactly supported measure-valued solutions with some a priori uniform bounds (in form of the specific form of the barycenter and the uniform bound in the quasiconvex envelope) if they are generated by weak solutions. However, the necessary conditions from above do not provide us with such bounds. This gap, originates from the breakdown of the Calder\'on-Zygmund estimate $\|u\|_{W^{2,p}}\leq C_p\|z\|_{L^p}$ in the case $p=\infty$ for $\mathcal{B}_Eu=z$, cf.~e.g.~\cite{LK64}. Thus, even if we were able to infer an $L^{\infty}$-bounded $\mathcal{A}_E$-free sequence generating the lifted measure, e.g.~by the results of \cite{BGS21}, the uniform $L^{\infty}$-bounds cannot be carried over to the level of potentials. One needs to overcome those technical difficulties for generalizing our proof to a full characterization of compactly supported measure-valued solutions generated by weak solutions. At the current state, this seems to be quite difficult, and it might even be possible that a new strategy of proof is needed for this.\\
	Note that one could argue that physically relevant measure-valued solutions are in fact those who come from a sequence of weak solutions or a vanishing viscosity sequence, cf.~\cite{BTW12} and \cite{GW20}.\\
	In the case of an admissible measure-valued solution $\nu$ our assumption $Q=\frac{2}{d}\langle\nu,e\rangle$ on the support of $\nu$ and the continuity of $\langle\nu,e\rangle$ are certainly too strong, since there is no reason for a general admissible measure-valued solution to fulfill them. Weakening these conditions might be worth some further study.
\end{rem}

\end{document}